\documentclass{gNST2e}
\usepackage[english]{babel}
\usepackage{txfonts}
\usepackage{dsfont}
\usepackage{array}
\usepackage[colorlinks,citecolor=blue,urlcolor=blue]{hyperref}
\usepackage{xcolor}
\usepackage{epsfig}
\usepackage{epstopdf}
\usepackage{subfigure}

\numberwithin{equation}{section}
\theoremstyle{plain}
\newtheorem{theorem}{Theorem}[section]

\newtheorem{lemma}[theorem]{Lemma}
\newtheorem{prop}[theorem]{Proposition}
\newcommand{\R}{\mathbb R}
\newcommand{\E}{\mathbb E}

\newcommand{\C}{\mathrm Cov}
\theoremstyle{remark}
\newtheorem{remark}{Remark}

\theoremstyle{definition}

\newdimen\AAdi%
\newbox\AAbo%
\def\AAk#1#2{\setbox\AAbo=\hbox{#2}\AAdi=\wd\AAbo\kern#1\AAdi{}}%

\begin{document}



\title{\textit{Nonparametric local linear estimation of the relative error regression function for censorship model}}

\author{
\name{F. Bouhadjera\textsuperscript{a,b}$^{\ast}$\thanks{$^\ast$Corresponding author. Email: feriel.bouhadjera@etu.univ-littoral.fr},
E. Ould Sa\"id\textsuperscript{c}}
\affil{\textsuperscript{a}UBMA, LaPS. BP 12, 23000 Annaba, Alg\'erie.; \textsuperscript{b}ULCO, LMPA, 50 Rue Ferdinand Buisson, Calais, 62228, France; \textsuperscript{c}ULCO, LMPA,  IUT de Calais. 19, rue Louis David. Calais, 62228, France.}
}

\maketitle

\begin{abstract}
In this paper, we built a new nonparametric regression estimator with  the local linear method by using  the mean squared relative error as a loss function when  the data are subject to random right censoring. We establish the uniform almost sure consistency with rate over a compact set of the proposed estimator.
Some simulations are given to show the asymptotic behavior of the estimate in different cases.
\end{abstract}


\begin{classcode}
\end{classcode}

{\abstractfont\centerline{\bfseries Table of contents}\vspace{12pt}
\hbox to \textwidth{\hsize\textwidth\vbox{\hsize24pc
\hspace*{-12pt} {1.}  Introduction\\
{2.}    The model \\
{3.}    Hypotheses and main results\\
\hspace*{10pt}{3.1.}  Comments on the Hypotheses\\
{4.}    Numerical study\\
{5.}    Proofs and auxiliary results\\
{6.}    Appendix\\
{7.}    Conclusion\\}}}

\section{Introduction}
When it comes to analyzing the dependence between two random variables (r.v.), regression models have appeared as a common and flexible tool in various disciplines, such as biology, medicine, economics, insurance. Consider a random vector $(T,X)$ taking values in $\R^+_* \times \R$ where $T$ is the interest r.v. with unknown distribution function (d.f.) $F$ and $X$ is the covariate considered having a  density function $f(\cdot)$. In practice, it is well-known that we have to study the association between covariates and responses according to the following relation:
\begin{equation*}
T=\mu(X)+\varepsilon
\end{equation*}
where $\mu(X)=\E[T|X]$ denotes the regression function which appears as a quantity that contains all the information about the dependence structure and $\varepsilon$ is the unobservable error term independent of $X$. Generally $\mu(X)$ is obtained by the minimization of $\E[(T-\mu(X))^2|X]$. The resulting estimator enjoys some important optimality, such as simplicity, flexibility, and consistency. However, this last loss function is inefficient to the presence of outliers in data, which is a common case in practical situations. \\
The aim of the present paper is to propose a new approach which  reduce these drawbacks. Relative error estimation has been recently used in regression analysis as an alternative to the restrictions imposed by the classical regression approach, which consist by considering the estimation of the regression function $\mu$ by minimizing  the following mean squared relative error loss function, that is,  for $T>0$
\begin{equation}\label{criter_RER}
\displaystyle \E\left[\left(\frac{(T-\mu(X))}{T}\right)^2\Big|X \right].
\end{equation} 
This criterium has been widely studied for parametric models, we refer to \cite{Chen2010} for a discussion about the previous works and \cite{Hirose2018} for a real example on the electricity consumption. When the first two conditional inverse moments of $T$ given $X$ are finite, \cite{Park1998} showed that the solution of (\ref{criter_RER}), for any fixed $x$, is given by the following ratio 
\begin{equation}\label{mu}
\displaystyle \mu(x)=\frac{\E[T^{-1}|X=x]}{\E[T^{-2}|X=x]}=:\frac{\mu_1(x)}{\mu_2(x)}
\end{equation}
where $\displaystyle \mu_{\ell}(x)=r_{\ell}(x)/f(x) $ and $\displaystyle r_{\ell}(x)=\int_{\R} t^{-\ell}f_{T,X}(t,x)dt$ for $\ell=1,2$ with $\displaystyle f_{T,X}(\cdot,\cdot)$ and $\displaystyle f_{X}(\cdot)$ are  the joint and  marginal density of the couple $(T,X)$ and $X$ respectively. For recent works, there have been some literature devoted to the relative error regression (RER) methods for complete data. \cite{Chahad2017} considered the estimation of the regression function for a functional explanatory variable while \cite{Attouch2017} have looked to the case where the data are from a strictly stationary spacial process. \cite{Thiam2018} constructed an estimator based in a deconvolution problem. \cite{Hu2019} established the consistency and the asymptotic normality of the regression function based on a least product relative error. \\  
It is well-known that the local linear method has several advantages over the classical kernel smoothing. In particular, it allows to reduce the bias term and avoid the boundary effects. The local linear smoother is not only superior to the popular kernel regression estimator, but also it is the best among all linear smothers, including those produced by orthogonal series and spline methods. A detailed introduction on the importance of the local linear approach can be found in \cite{Fan1992}, \cite{Fan1996} for the univariate case and \cite{Fan2003} for the multivariate case. For recent works on local linear method, we refer to \cite{Jones2008} for independent data and \cite{ElGhouch2008, ElGhouch2009} for regression and quantile regression respectively in the dependent framework.\\
All these works concern the complete data except the last two articles. In many situations, the data can not  be observed completely.  Important examples are the survival time of patients or the unemployment time and many others in different fields. A frequent problem in survival analysis is right-censoring, which may be due to different causes: the loss of some subjects under study, the end of the follow up period. Examples of situations where this kind of data occur can be found in \cite{Klein2006}. \\
In this paper, we suggest a new estimator based on the local linear method of the nonparametric relative error regression (LLRER) estimator when the data are censored. 
We extend the work of \cite{Jones2008} to the censoring framework by stating a strong result. We point out that in the last paper, only a pointwise of the bias and variance terms have been investigated. We establish that the new estimator is uniformly almost sure consistent with rate over a compact set under appropriate conditions. Simulation experiments emphasize that the LLRER, is highly competitive to the existing estimators for regression function. To the best of our knowledge, this problem is open up to now and there is no analogous result.\\
This paper is organized as follows. The general idea of the local linear fit of the mean squared relative error regression function in the censoring framework is described in \hyperref[sect 2]{Section 2}. Assumptions and theoretical results are given in \hyperref[sect 3]{Section 3} and some simulation results that illustrates the performance of the proposed procedure are given in \hyperref[sect 4]{Section 4}. Finally, \hyperref[sect 5]{Section 5} is devoted to auxiliary results and technical details.	
\section{The model}\label{sect 2}
According to the right-censoring model, instead of observing $T$ we only observe $(Y,\delta)$ where $\displaystyle Y=\min(T,C)$ and $\displaystyle \delta=\mathds{1}_{\{T \leq C\}}$, here $\displaystyle \mathds{1}(\cdot)$ is the indicator function. The r.v. $C$ represent the censoring time which is independent of $T$ and with d.f. $G$. The observed data becomes $(Y,\delta,X)$. From now on, we will always make the following assumption:
\begin{equation}\label{indep}
(T,X)\;\; \textrm{and}\;\; C \;\; \textrm{are independent.}
\end{equation}
This assumption is required to make the estimation of the censoring distribution easier; However, it is reasonable only when the censoring is not associated to the characteristic of the individuals under study. Let $\{(Y_i,\delta_i,X_i),\;i=1,\dots,n\}$ be $n$ independent and identically distributed vectors as $(Y,\delta,X)$. Our main aim is to estimate the RER function defined in (\ref{mu}) using the local linear fit. The extension of nonparametric local linear procedures to the censored framework requires to replace the unavailable data by a suitable construction of the observed data given by 
\begin{equation}\label{synt_data}
\displaystyle T^{\star,-\ell}_i=\frac{\delta_i Y_i^{-\ell}}{\overline{G}(Y_i)} \quad\quad \text{for} \quad 1 \leq i \leq n
\end{equation}
where $\overline{G}(\cdot)=1-G(\cdot)$ denotes the survival function of the r.v. $C$. The later are called "synthetic data" and permits to consider the effect of censoring in the distribution (for more details, we refer to \cite{Carbonez1995} and \cite{Kohler2002}). In this spirit, based on this construction of the data, using the conditional expectation property and under the Assumption (\ref{indep}), for $\ell=1,2$ we have
\begin{equation*}\label{synt}
\begin{aligned}
\displaystyle 	\E[T^{\star,-\ell}_1|X_1]&=\E\left[\frac{\delta_1Y_1^{-\ell}}{\overline{G}(Y_1)}\Big|X_1\right]\\
&\displaystyle =\E\left[\frac{T_1^{-\ell}}{\overline{G}(T_1)}\E\left[\mathds{1}_{\{T_1 \leq C_1\}}|T_1\right]X_1\right]\\
&\displaystyle =\E[T_1^{-\ell}|X_1].
\end{aligned}
\end{equation*}
\noindent Modeling by the local linear method (see \cite{Fan1992}), assumes that the twice derivative of $\mu(x)$ at the point $x$ exists and is continuous, so that $\mu(X)$ can be approximated by a linear function that is, $\mu(X) \approx \mu(x)+\mu^{\prime}(x)(X-x)=:\beta_1+\beta_2(X-x)$. Then, the RER function (\ref{mu}) is estimated as the solution of the following optimization problem : 
\begin{equation}\label{argmin}
\arg \min_{\beta_1,\beta_2} \left\{\sum_{i=1}^{n} T^{\star,-2}_i(T^\star_i-\beta_1-\beta_2(X_i-x))^2 K_h(X_i-x)\right\}
\end{equation}
where $\displaystyle K_h(\cdot):=K\Big(\frac{\cdot}{h}\Big)$ is a kernel function appropriately chosen (Epanechnicov, Gaussian, $\ldots$ ) and $h:=h_n$ is a sequence of positive real numbers which converges to $0$ when $n$ goes to infinity. By elementary calculus, the solution of the least squares problem (\ref{argmin}) yields to 
\begin{equation}\label{mu_tilde}
\widetilde{\mu}(x)=\frac{\displaystyle \sum_{i,j=1}^{n}w_{i,j}(x)T^{\star}_j}{\displaystyle \sum_{i,j=1}^{n}w_{i,j}(x)}=:\frac{\displaystyle \widetilde{\mu}_{1}(x)}{\displaystyle \widetilde{\mu}_2(x)}
\end{equation}
where
\begin{equation}\label{w_tilde}
w_{i,j}(x)=(X_i-x)\left((X_i-x)-(X_j-x)\right) K_h(X_i-x)K_h(X_j-x) T^{\star,-2}_i T^{\star,-2}_j.
\end{equation}
Of course in data analysis, the survival function $\overline{G}(\cdot)$ is unknown and needs to be estimated. This can be done via Kaplan-Meier (KM) as  an estimator of $\overline{G}(\cdot)$ (see: \cite{Kaplan1958})
\begin{equation}\label{K-M}
\overline{G}_n(t)=
\left\{
\begin{array}{cl}
\displaystyle \prod_{i=1}^{n}{\left(1-\frac{1-\delta_{i}}{n-i+1}\right)}^{\mathds{1}_{\{Y_{i}\leq t\}}} & \quad \text{if} \quad t<Y_{(n)}, \\
0 & \quad \text{otherwise}
\end{array}
\right.
\end{equation}
where $Y_{(1)}\leq Y_{(2)}\leq \dots \leq Y_{(n)}$ are the order statistics of the $Y_{i}$ and $\delta_{i}$ is the indicator of non-censoring. The properties of $\overline{G}_n(t)$ have been studied by many authors. So, (\ref{synt_data}) becomes, for $1 \leq i \leq n$,
\begin{equation}\label{sythetic}
\widehat{T}^{*,-\ell}_i=\frac{\delta_i Y_i^{-\ell}}{\overline{G}_n(Y_i)}.
\end{equation}
Replacing (\ref{sythetic}) in (\ref{mu_tilde}) and (\ref{w_tilde})  we get  a feasible local linear estimator of the relative error regression  function (LLRER)  expressed as
\begin{equation}\label{mu_hat}
\widehat{\mu}(x)=\frac{\displaystyle \sum_{i,j=1}^{n}w_{i,j}(x)\widehat{T}^{\star}_j}{\displaystyle \sum_{i,j=1}^{n}w_{i,j}(x)}=:\frac{\displaystyle \widehat{\mu}_{1}(x)}{\displaystyle \widehat{\mu}_2(x)}
\end{equation}
where
\begin{equation}\label{w_hat}
w_{i,j}(x)=(X_i-x)\left((X_i-x)-(X_j-x)\right) K_h(X_i-x)K_h(X_j-x) \widehat{T}^{\star,-2}_i \widehat{T}^{\star,-2}_j.
\end{equation}
\begin{remark} 
	In what follows, we will adopt the convention $0/0=0$ in such a case that if, for example, $\widehat{\mu}_1(\cdot)=0$ and $\widehat{\mu}_2(\cdot)=0$, the ratio $\widehat{\mu}_1(\cdot)/\widehat{\mu}_1(\cdot)$ in (\ref{mu_hat}) will be interpreted as zero.
\end{remark}
\noindent Throughout this paper, we denote by $\tau_F:=\sup\{x: \overline{F}(x)>0 \}$ and $\tau_G:=\sup\{ x : \overline{G}(x)>0 \}$ be the right support endpoints of $\overline{F}$ and $\overline{G}$, respectively.  We assume that $\tau_F <\infty$, $\overline{G}(\tau_F)>0$ that implies $0<\tau_F \leq \tau_G$, which were also assumed in \cite{Guessoum2008}.
\begin{remark}
	In the simulation part, we will compare our estimator with the classical regression estimator using the local linear method (LLCR). The later is the solution of the following minimization problem:
	\begin{equation*}
	\arg \min_{\alpha,\; \beta} \left\{\sum_{i=1}^{n} \left(\widehat{T}^\star_i-\alpha-\beta(X_i-x)\right)^2 K_h(X_i-x)\right\}
	\end{equation*}
	for $\widehat{T}^\star$ in (\ref{sythetic}), which gives
	\begin{equation}\label{LLCR}
	m_n(x)=\frac{ \displaystyle \sum_{i,j=1}^{n} v_{i,j}(x)\widehat{T}^{\star}_j}{\displaystyle \sum_{i,j=1}^{n} v_{i,j}(x)}
	\end{equation}
	where 
	\[ v_{i,j}(x)=(X_i-x)\left((X_i-x)-(X_j-x)\right) K_h(X_i-x)K_h(X_j-x).\]
\end{remark}
\begin{remark} 
	\begin{itemize}
		\item[$1)$] We point out that for complete data, i.e. we replace $\widehat{T}^{\star}$ by $T$ in (\ref{mu_hat}) and (\ref{w_hat}), we obtain the estimator defined in \cite{Jones2008}.
		\item[$2)$]  Likewise, if we replace $\widehat{T}^{\star}$ by $T$ in (\ref{LLCR}), we obtain the estimator defined in \cite{Nadaraya1964} and \cite{Watson1964}. 
	\end{itemize}
\end{remark}
\begin{remark}
	A crucial point in censored regression is to extend the identifiability assumption on the independence of $T$ and $C$ defined in (\ref{indep}) to the case where the explanatory variables are present. In this spirit of KM estimator, one may impose that $T$ and $C$ are independent conditionally to $X$. Then, (\ref{sythetic}) becomes 
	\begin{equation}\label{conditionel}
	\widehat{T}^*_i= \frac{\delta_i Y_i}{\overline{G}_n(Y_i|X_i)} 
	\end{equation}
	where $\overline{G}_n(Y_i|X_i)$ is Beran's estimator of the survival conditional function of the r.v. $C$ given $X$, for more details see \cite{Beran1981}. The property of this estimator has been studied by \cite{Dabrowska1987} and \cite{Dabrowska1989}. Replacing (\ref{conditionel}) in (\ref{mu_hat}) and (\ref{w_hat}) we obtain a feasible estimator of the LLRER function $\mu(\cdot)$.
\end{remark}
\begin{remark}\label{remark2.5}
	A frequently used bandwidth selection technique is the cross-validation method, which choose $h$ to minimize 
	\begin{equation}\label{CV}
	\sum_{i=1}^n \left(\widehat{T}^\star_i-\widehat{\mu}_{-i}(X_i)\right)^2
	\end{equation}
	where $\widehat{\mu}_{-i}(\cdot)$ is the LLRER estimator defined in (\ref{mu_hat}) without using the $i^{th}$ observation $(X_i,T_i)$.
\end{remark}
\section{Hypotheses and main results}\label{sect 3}
We will use the following notation ${\mathcal{C}}$ to refer to a compact set of ${\mathcal{C}}_{0}$ where ${\mathcal{C}}_{0}=\left\{x \in \mathbb{R}^+ / f(x) >0 \right\}$ is an open set. Furthermore, when no confusion is possible, we will denote by $C$ any generic positive constant and we assume that 
\begin{equation}\label{bounded}
\forall T>0, \exists\; C, \;\;\text{such that} \;\; |T|^{-\ell} \leq C.
\end{equation}
\begin{itemize}
	\item[\bf H1] The bandwidth $h$ satisfies $ \displaystyle\lim_{n \rightarrow \infty} h=0, \quad \lim_{n \rightarrow \infty} n h=+\infty,\quad \lim_{n\rightarrow\infty} \frac{\log n}{nh}=0$.\label{H1}
	\item[\bf H2] The kernel $K(\cdot)$ is bounded, symmetric non-negative function on $\mathcal{C}$.\label{H2}
	\begin{itemize}
		\item[\bf i.] $\int t^j K(t)dt<\infty$, for $j=2,3$.\label{H2i}
		\item[\bf ii.] $\int t^j K^2(t)dt<\infty$ for $j=2,3$.\label{H2ii}
	\end{itemize}
	\item[\bf H3] The density function $f(\cdot)$ is continuously differentiable and  $\displaystyle \sup_{x \in \mathcal{C}} |f^{\prime}(x)| <+\infty$. \label{H3} 
	\item[\bf H4] The function $r_\varrho(x)=\int t^{-\varrho}f_{T,X}(t,x)dt$ for $\varrho=1,2,3,4$ is continuously, differentiable and $\displaystyle \sup_{x \in \mathcal{C}} |r_\varrho^{\prime}(x)| <+\infty$. \label{H4}
	\item[\bf H5] The function \[\upsilon_{\ell,k}(x)=\int t^{-\ell k}f_{T,X}(t,x)dt,\ell=1,2\; \text{and}\; 0\leq k\leq \nu\] 
	is continuously differentiable and $\displaystyle \sup_{x \in \mathcal{C}} |\upsilon_{\ell,k}^{\prime}(x)| <+\infty$. \label{H5}
\end{itemize} 
\subsection{Comments on the Hypotheses:}
The hypothesis \hyperref[H1]{ H1} concern the bandwidth and is very common in nonparametric estimation. The hypothesis \hyperref[H2]{H2} regards the Kernel $K$ and are needed for the convergence of the estimator. Analogous hypotheses on the kernel has been also made by \cite{Fan1992}. The hypothesis \hyperref[H3]{H3} deals with the density function $f(\cdot)$. The hypothesis \hyperref[H4]{ H4} and \hyperref[H5]{ H5} are regularity conditions for $r_{\varrho}(\cdot)$ and $\upsilon_{\ell,k}(\cdot)$ respectively for different value of $\ell,\varrho$ and $k$. 
\begin{theorem}\label{theo1}
	Under Hypotheses \hyperref[H1]{H1-H5}, for $n$ large enough, we have
	\begin{equation*}
	\sup_{x\in {\cal C}} |\widehat{\mu}(x)-\mu(x)|=\ \text{O}\Big(h^3\Big) + \text{O}_{a.s.}\left(\sqrt{\frac{\log n}{nh}}\right).
	\end{equation*}
\end{theorem}
\noindent  The proof of the \hyperref[theo1]{Theorem 1} is made up on the following decomposition:
\begin{multline*}
\widehat{\mu}(x)-\mu(x) =\frac{1}{\widehat{\mu}_{2}(x)} \left\{ \widehat{\mu}_{1}(x)-\widetilde{\mu}_{1}(x)+\widetilde{\mu}_{1}(x)-\E[\widetilde{\mu}_{1}(x)]+\E[\widetilde{\mu}_{1}(x)]-r_1(x)r_2(x) \right.\\ 
+ \left. \mu(x) \left\{ r^2_2(x)-\E[\widetilde{\mu}_{2}(x)]+\E[\widetilde{\mu}_{2}(x)]-\widetilde{\mu}_{2}(x)+\widetilde{\mu}_{2}(x)-\widehat{\mu}_2(x)\right\}\right\}.
\end{multline*}
Remark that  by   Hypothesis \hyperref[H4]{ H4}  and condition (\ref{bounded}), there exists $\eta>0$ such that $\displaystyle \sup_{x\in {\cal C}}| r_2(x)|\leq \eta$. Then,
by triangle inequality, we have
\begin{equation*}
\begin{aligned}
\sup_{x \in \mathcal{C}} \big| \widehat{\mu}(x)-\mu(x) \big|
&\leq \frac{1}{\eta^2 - \displaystyle \sup_{x \in \mathcal{C}} \big|\widehat{\mu}_{2}(x)-r_2^2(x)\big|}
\left\{ \sup_{x \in \mathcal{C}}\big|\widehat{\mu}_{1}(x)-\widetilde{\mu}_{1}(x) \big|+\sup_{x \in \mathcal{C}} \big|\widetilde{\mu}_{1}(x)-\E[\widetilde{\mu}_{1}(x)] \big| \right. \\
& + \sup_{x \in \mathcal{C}} \big|\E[\widetilde{\mu}_{1}(x)]-r_1(x)r_2(x)\big|+ \sup_{x \in \mathcal{C}} \big|\mu(x)\big| \left\{ \sup_{x \in \mathcal{C}} \big| \E[\widetilde{\mu}_2(x)]- r^2_2(x) \big|\right.  \\ 
& + \left. \left. \sup_{x \in \mathcal{C}} \big|\widetilde{\mu}_2(x)-\E[\widetilde{\mu}_2(x)] \big|+\sup_{x \in \mathcal{C}}\big|\widehat{\mu}_2(x)-\widetilde{\mu}_2(x)\big| \right\} \right\}.
\end{aligned}
\end{equation*}
The proof will be achieved with the following propositions. 
\begin{prop}\label{prop1}
	Under Hypotheses \hyperref[H1]{H1}, \hyperref[H2i]{H2 i)}, \hyperref[H3]{H3} and \hyperref[H4]{H4}, for $\ell=1,2$,  for $n$ large enough, we have 
	\begin{equation*}
	\displaystyle \sup_{x \in \mathcal{C}}\left|\widehat{\mu}_{\ell}(x)-\widetilde{\mu}_{\ell}(x)\right|=\text{O}_{a.s.}\left(\sqrt{\frac{\log \log n}{n}}\right).
	\end{equation*}
\end{prop}
\begin{prop}\label{prop2}
	Under Hypotheses \hyperref[H1]{H1}, \hyperref[H2i]{H2 i)}, \hyperref[H3]{H3}, \hyperref[H4]{H4} and \hyperref[H5]{H5}, for $\ell=1,2$, for $n$ large enough, we have
	\begin{equation*}
	\sup_{x \in \mathcal{C}}\big|\widetilde{\mu}_{\ell}(x)-\E[\widetilde{\mu}_{\ell}(x)]\big|=\text{O}_{a.s.} \left(\sqrt{\frac{\log n }{n h }}\right).
	\end{equation*}
\end{prop}
\begin{prop}\label{prop3}
	Under Hypotheses \hyperref[H1]{H1}, \hyperref[H2]{H2} and \hyperref[H4]{H4}, for $\ell=1,2,$ for $n$ large enough, we have
	\begin{equation*}
	\sup_{x \in \mathcal{C}}\left|\E[\widetilde{\mu}_{\ell}(x)]-r_\ell(x)r_2(x)\right|=\text{O}\left(h^3\right).
	\end{equation*}
\end{prop}
\section{Numerical study}\label{sect 4}
To evaluate the quality of this method, we perform several simulations of the proposed estimator $\widehat{\mu}(\cdot)$ with different level of censoring. For that, we generate the data as follows:
\begin{description}
	\item[{\bf Inputs:}] Generate $n$ i.i.d. \{$X_i \leadsto \mathcal{N}(0,1)$, $C_i \leadsto \mathcal{N}(3+c,1)$ and $\epsilon_i \leadsto \mathcal{N}(0,1)$\} for $1 \leq i \leq n$ where $c$ is a constant that adjusts the percentage of censoring (C.P.).
	\begin{description}
		\item[Step 1 :]	 Calculate the interest variable $T_i=2X_i+1+0.2\,\epsilon_i$ where $X_i$ and $\varepsilon_i$ are independent.
		\item[Step 2 :]  Compute the observed data $\{T^{\star}_i,1 \leq i \leq n\}$ from (\ref{sythetic}) with the KM estimator from (\ref{K-M}).
		\item[Step 3 :]  We employ the Gaussian Kernel. Furthermore, we apply the cross-validation method (see : \hyperref[remark2.5]{Remark 2.5}) to choose the bandwidth. For a predetermined sequence of $h$'s from a wide range ($0.01$ to $2$) with an increment $0.01$, we choose the optimal bandwidth ($h_{opt}$) that minimize the cross-validation criterium (\ref{CV}).
	\end{description}
	\item[{\bf Ouputs:}] Compute the LLRER estimator from (\ref{mu_hat}) for $x \in [1,4]$ and $h_{opt}$.
\end{description}
In all the simulation study, we use the following proposition of \cite{Port1994} which permit to calculate the theoretical RER function (see formula (\ref{TC}) below).
\begin{prop}
	Let $q_1(X)$ and $q_2(X)$ be two random variable with means: $\mu_1$ and $\mu_2$ and variances: $v_1$ and $v_2$ respectively, and covariance $v_{12}$. Let $(X_i)_{1 \leq i \leq n}$ be an i.i.d. sequence of r.v. and defined by 
	\[\displaystyle \widehat{\Sigma}_1= \frac{1}{n} \sum_{i=1}^{n} q_1(X_i) \;\; \text{and} \;\; \widehat{\Sigma}_2= \frac{1}{n} \sum_{i=1}^{n} q_2(X_i)\]
	and $\displaystyle  \widehat{R}=\displaystyle{\frac{\widehat{\Sigma}_1}{\widehat{\Sigma}_2}}$ then the second order approximation of $\displaystyle \E[\widehat{R}]$ is 
	\[\displaystyle  \E[\widehat{R}] \approx \frac{\mu_1}{\mu_2}+ \frac{1}{n} \left( \frac{\mu_1 v_2}{\mu_2^3}-\frac{v_{12}}{\mu_2^2}\right).\]
\end{prop}
\noindent In the following figures, the solid line represent the theoretical curve (TC) of the RER function which is generated according to the following formula: 
\begin{equation}\label{TC}
m(x)=2x+1+0.04(2x+1)^{-1} \;\;\; \text{for} \;\;\; x \in [1,4]
\end{equation}
Furthermore, a comparative study with other existing kernel methods: the classical regression (CR) estimator defined in \cite{Guessoum2008} by 
\begin{equation*}
\widehat{m}(x)=\frac{\displaystyle \sum_{i=1}^n \widehat{T}^\star_{i} K_h(X_i-x)}{\displaystyle \sum_{i=1}^n K_h(X_i-x)}
\end{equation*}
and the local linear classical regression (LLCR) estimators defined in (\ref{LLCR}) was carried out.
\subsection{Effect of sample size:} We plot the true RER curve (TC) together with the LLRER estimator in \hyperref[figure1]{Figure 1}. We can see that the quality of fit is better when $n$ rises.
\begin{figure}[!h]
	\hspace*{-0.5cm}
	\begin{minipage}[c]{.26\linewidth}
		\includegraphics[height=2in, width=2in]{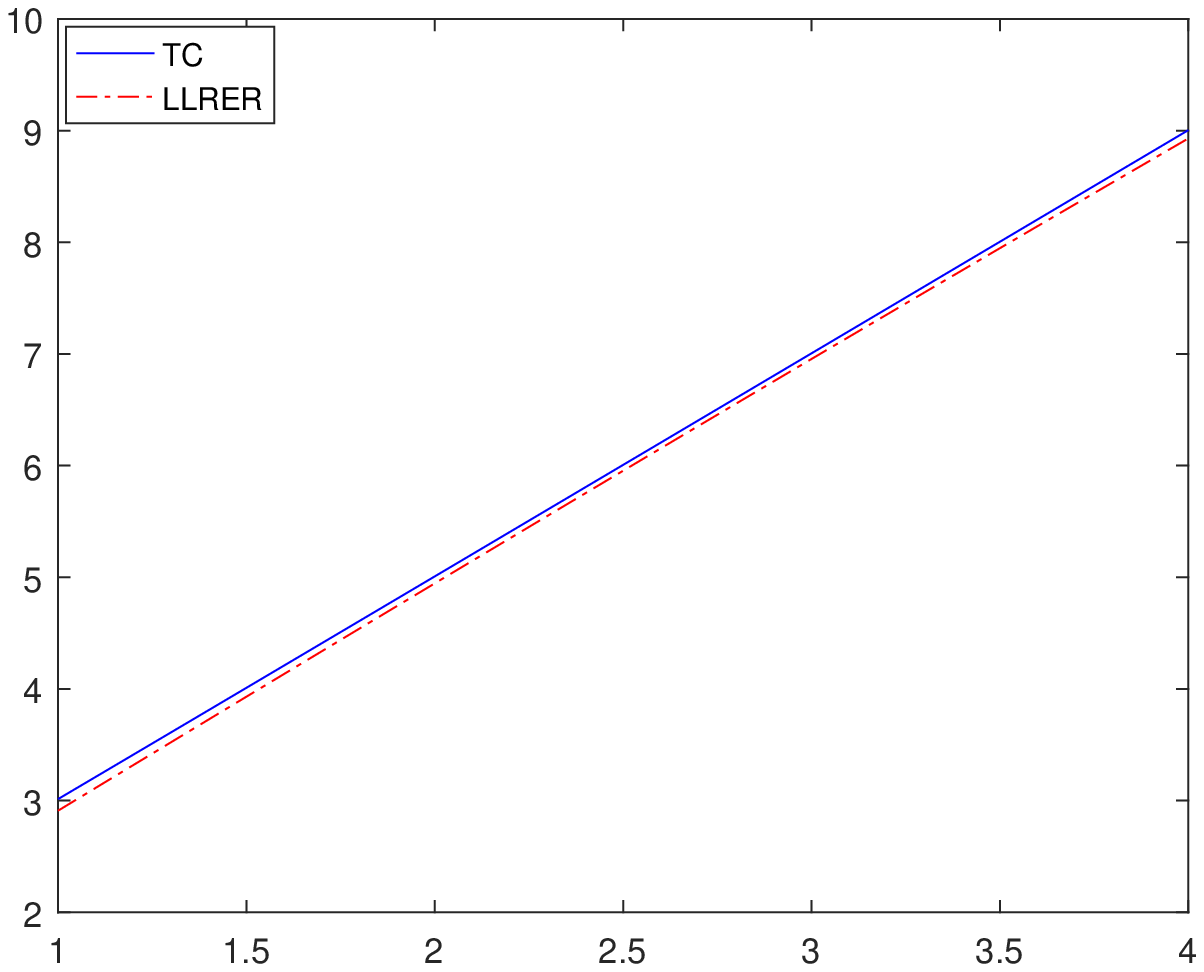}
	\end{minipage} \hfill
	\begin{minipage}[c]{.26\linewidth}
		\includegraphics[height=2in, width=2in]{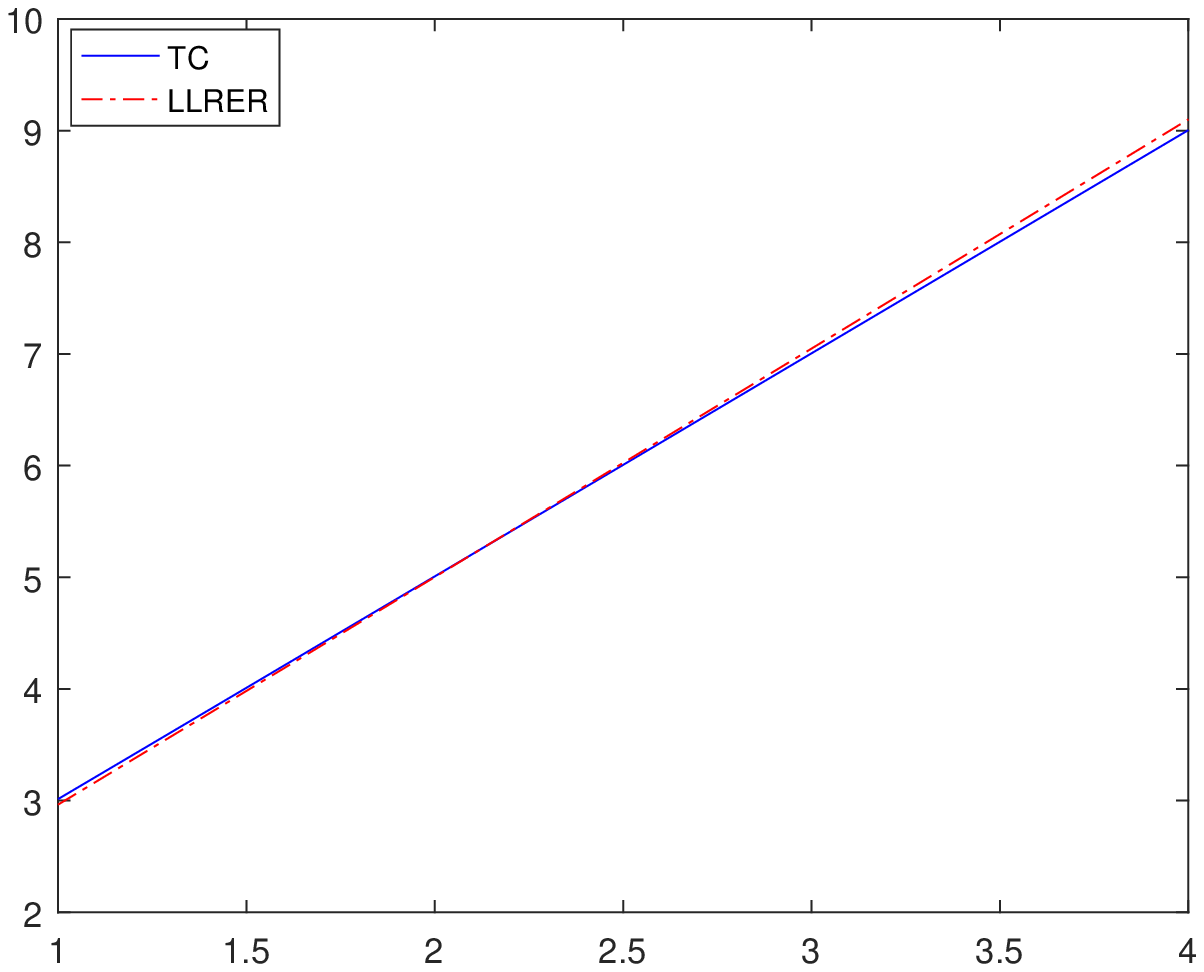}
	\end{minipage} \hfill
	\begin{minipage}[c]{.26\linewidth}
		\includegraphics[height=2in, width=2in]{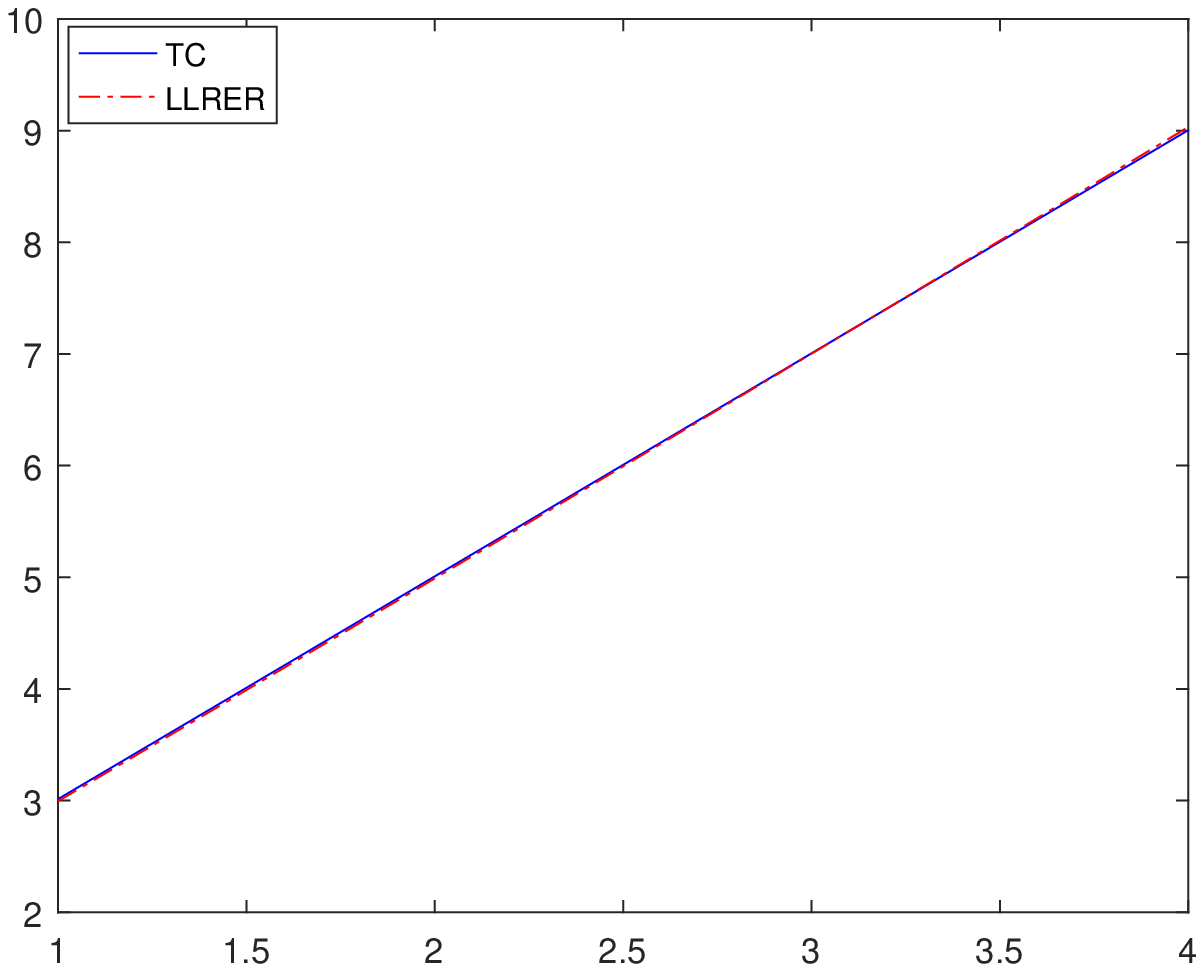}
	\end{minipage}\hfill\hfill
	\caption{\textcolor{blue}{$\mu(\cdot)$}, \textcolor{red}{$\widehat{\mu}(\cdot)$} with C.P.$\approx 65\%$ for $n=100,300,$ and $500$ respectively.}\label{figure1}
\end{figure}
\vspace*{1in}

\subsection{Effect of C.P.:} From \hyperref[figure2]{Figure 2}, it can be seen for a fixed sample size that the LLRER estimator quality is a little bit affected by the percentage of observed data. 
\begin{figure}[!h]
	\hspace*{-0.5cm}
	\begin{minipage}[c]{.26\linewidth}
		\includegraphics[height=2in, width=2in]{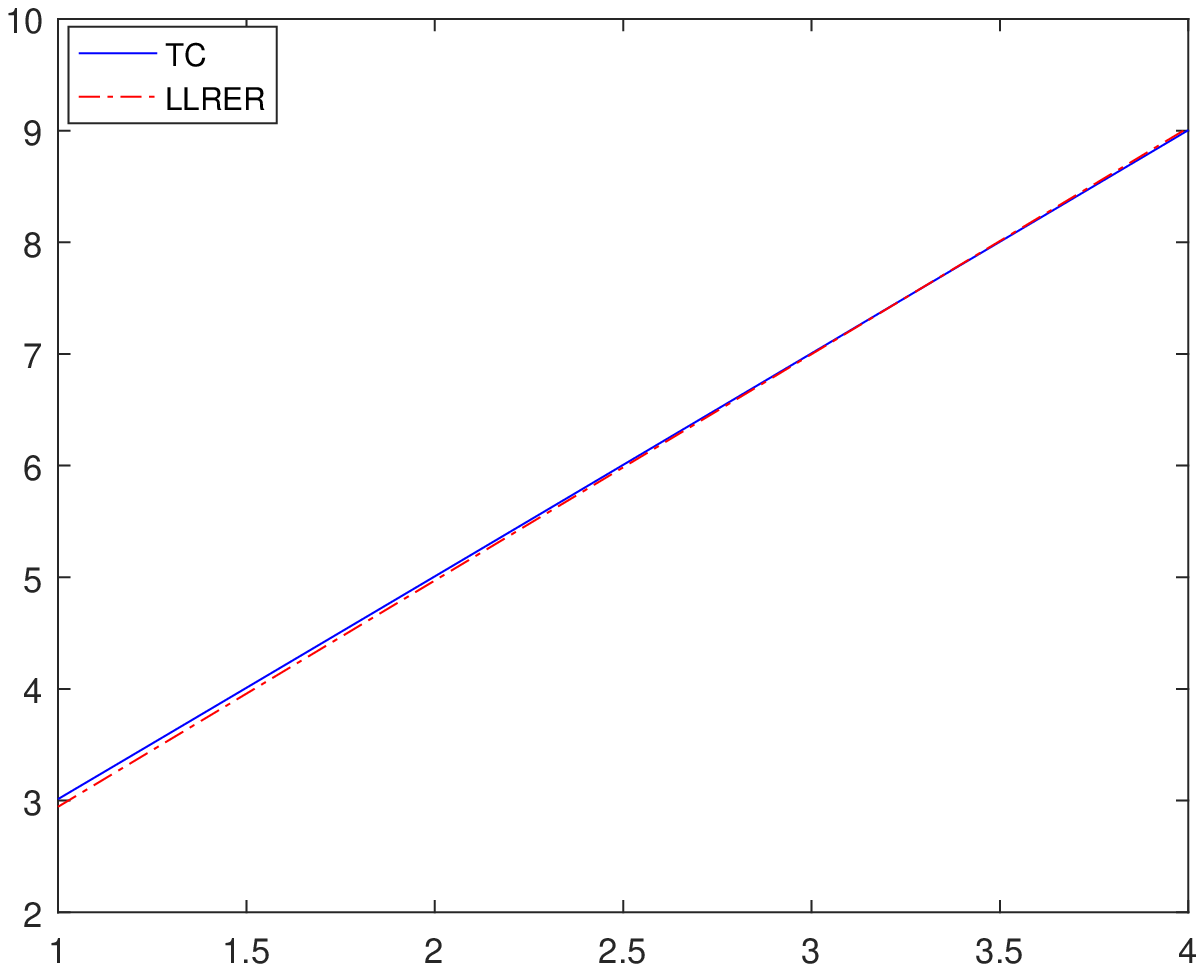}
	\end{minipage} \hfill
	\begin{minipage}[c]{.26\linewidth}
		\includegraphics[height=2in, width=2in]{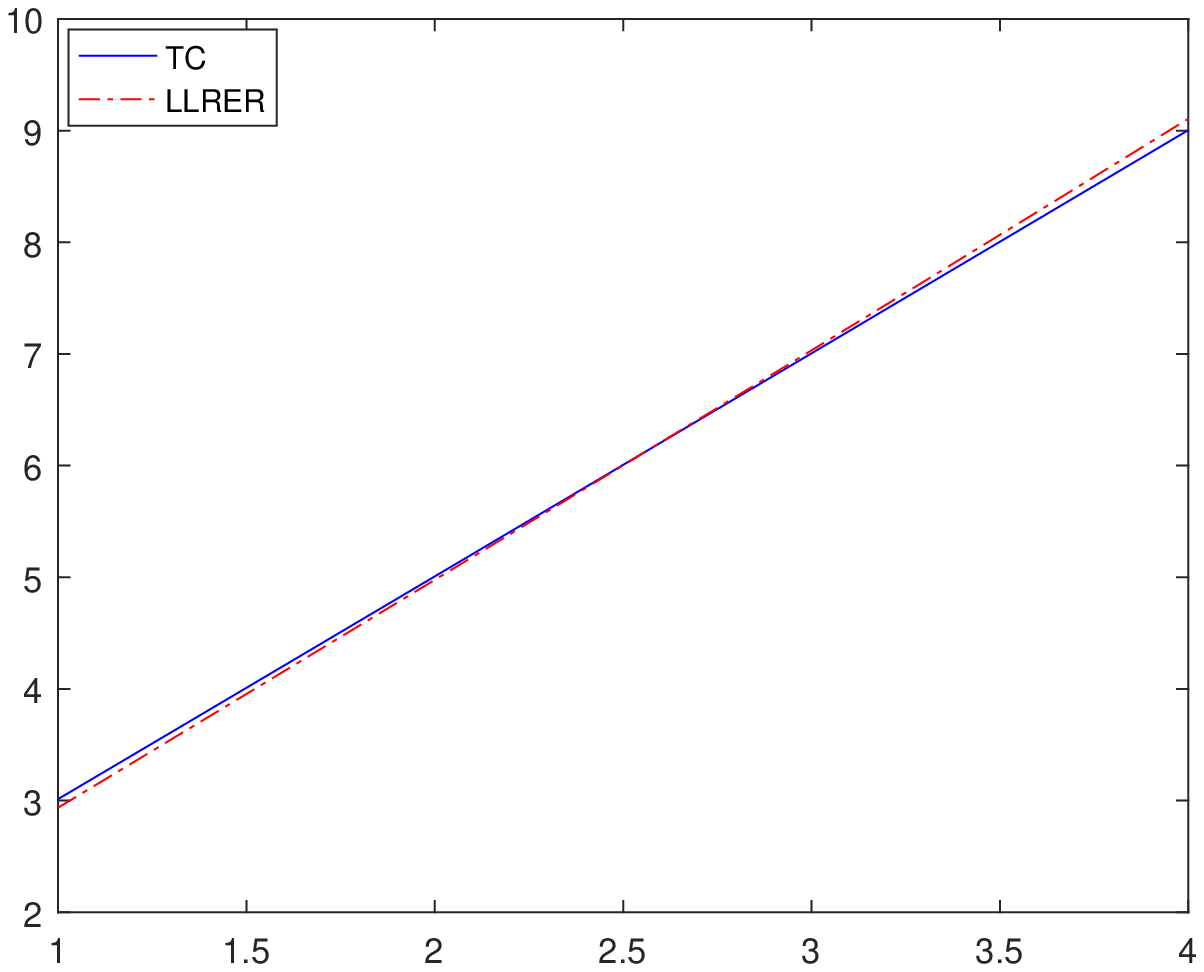}
	\end{minipage} \hfill
	\begin{minipage}[c]{.26\linewidth}
		\includegraphics[height=2in, width=2in]{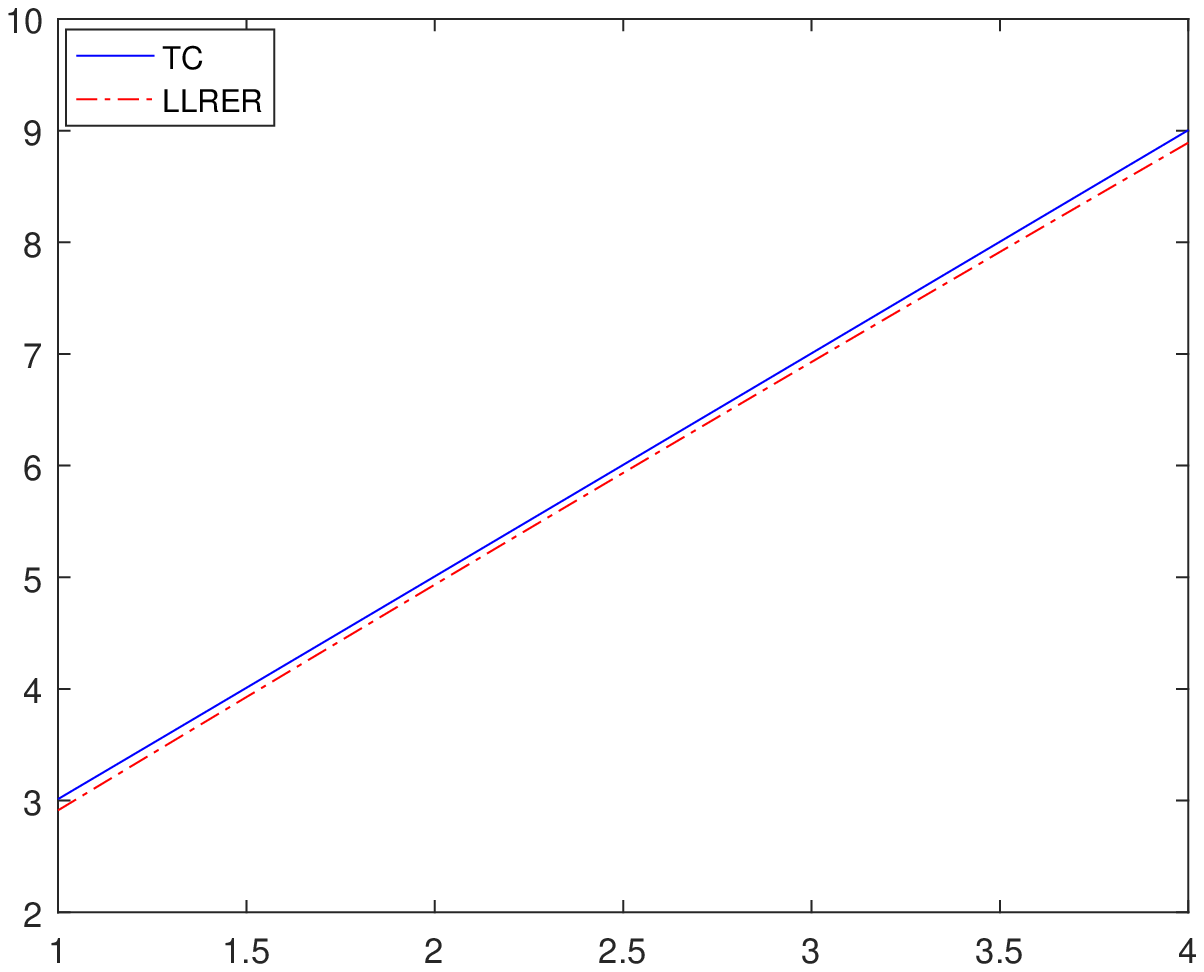}
	\end{minipage}\hfill\hfill
	\caption{\textcolor{blue}{$\mu(\cdot)$}, \textcolor{red}{$\widehat{\mu}(\cdot)$} with $n=300$ for C.P.$\approx 35, 50, \;\text{and}\; 70\%$ respectively.}\label{figure2}
\end{figure}
\subsection{Effect of outliers:} 
In order to assess the robustness to outliers of our new estimator, we generate samples of size $n = 300$ and multiply the values of 15 among them by a multiplying coefficient (M.C.).
We can observe that the quality of fit decreases as the value of M.C. increases but remains consistent.
\begin{figure}[!h]
	\hspace*{-0.5cm}
	\begin{minipage}[c]{.26\linewidth}
		\includegraphics[height=2in, width=2in]{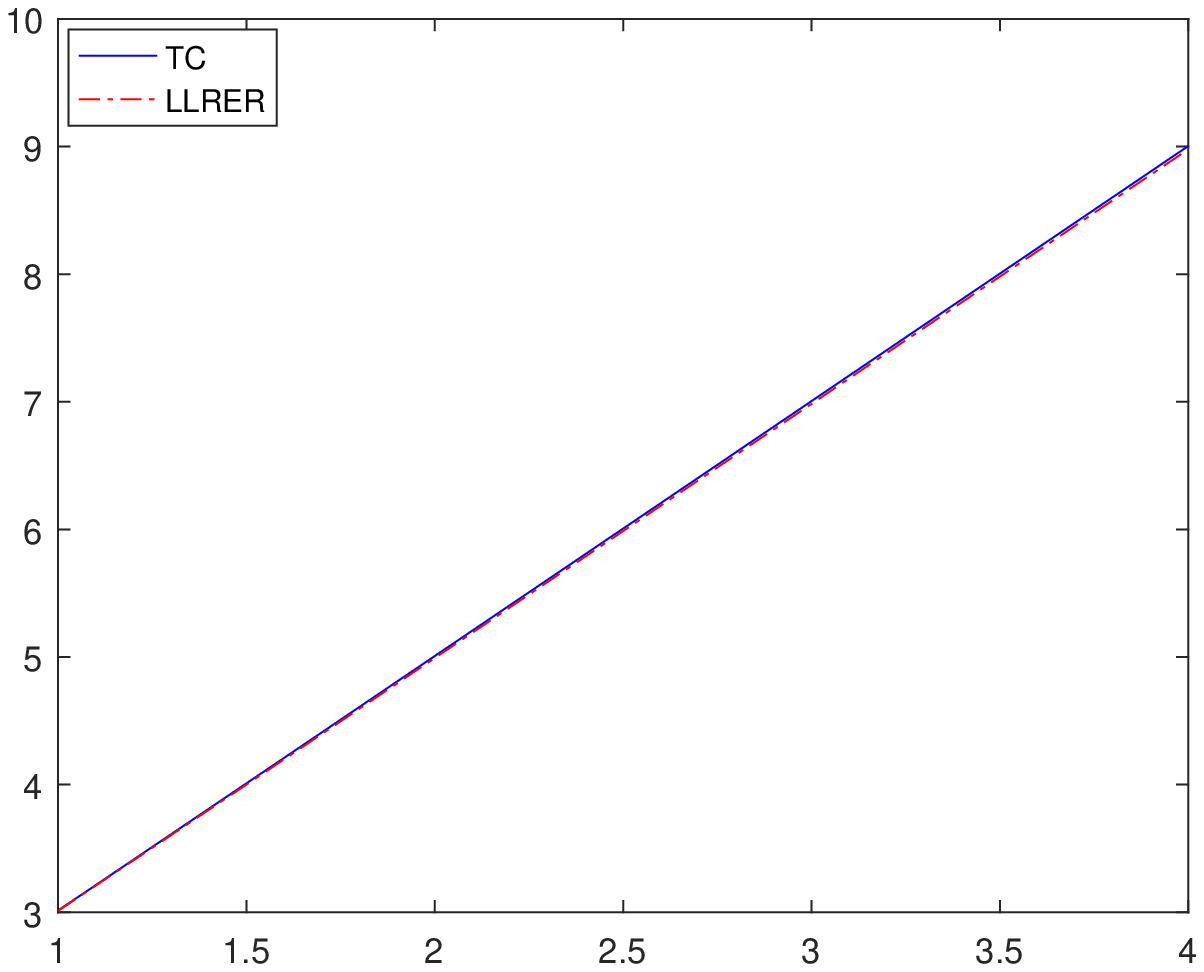}
	\end{minipage} \hfill
	\begin{minipage}[c]{.26\linewidth}
		\includegraphics[height=2in, width=2in]{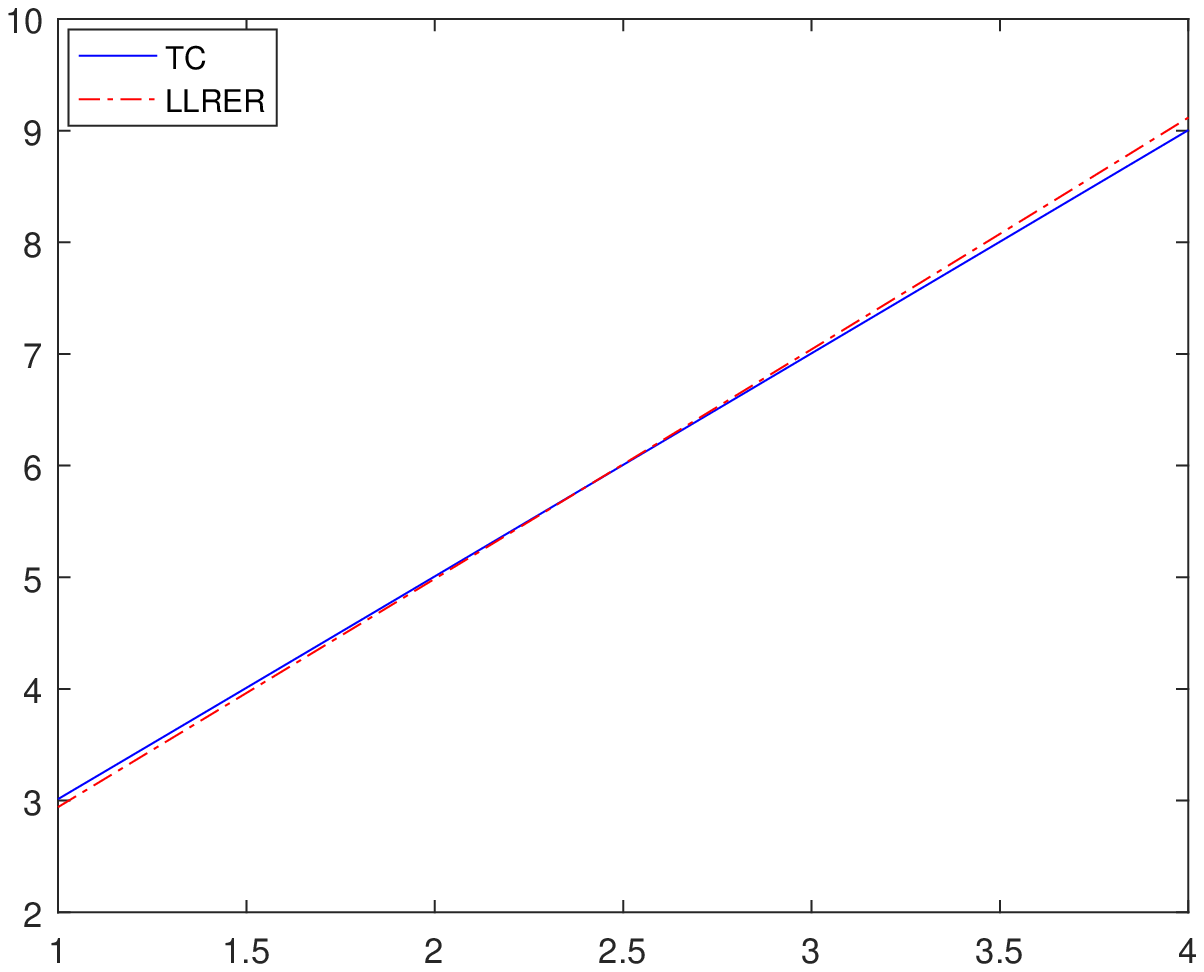}
	\end{minipage} \hfill
	\begin{minipage}[c]{.26\linewidth}
		\includegraphics[height=2in, width=2in]{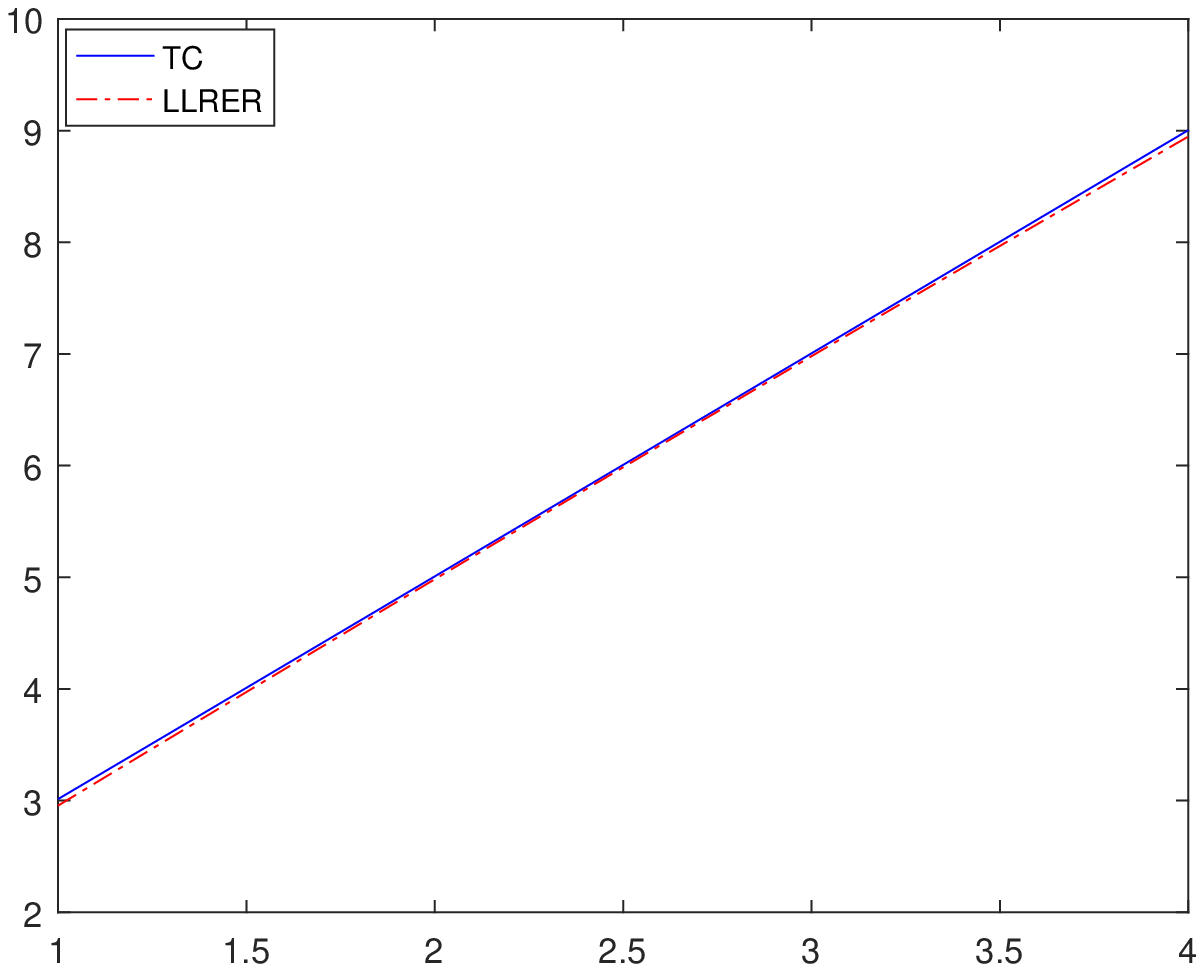}
	\end{minipage}\hfill\hfill
	\caption{\textcolor{blue}{$\mu(\cdot)$}, \textcolor{red}{$\widehat{\mu}(\cdot)$} with $n=300$ for C.P.$\approx 50\%$ and M.C.$=25,50,\;\text{and}\;100$ respectively.}\label{figure3}
\end{figure}
\subsection{Comparison to other kernel estimators:}
\vspace*{0.3cm}

\subsubsection{CR vs LLRER:}
\noindent  {\it Effect of C.P.:}\\
The proposed estimate shows an improvement over the CR estimate near the right tail where the data points are sparse and mostly uncensored. \hyperref[figure4]{Figure 4} shows that the LLRER estimator is much more robust to censoring than the CR, in particular for larger samples. 
\begin{figure}[!h]
	\begin{minipage}[c]{.26\linewidth}
		\includegraphics[height=2in, width=2in]{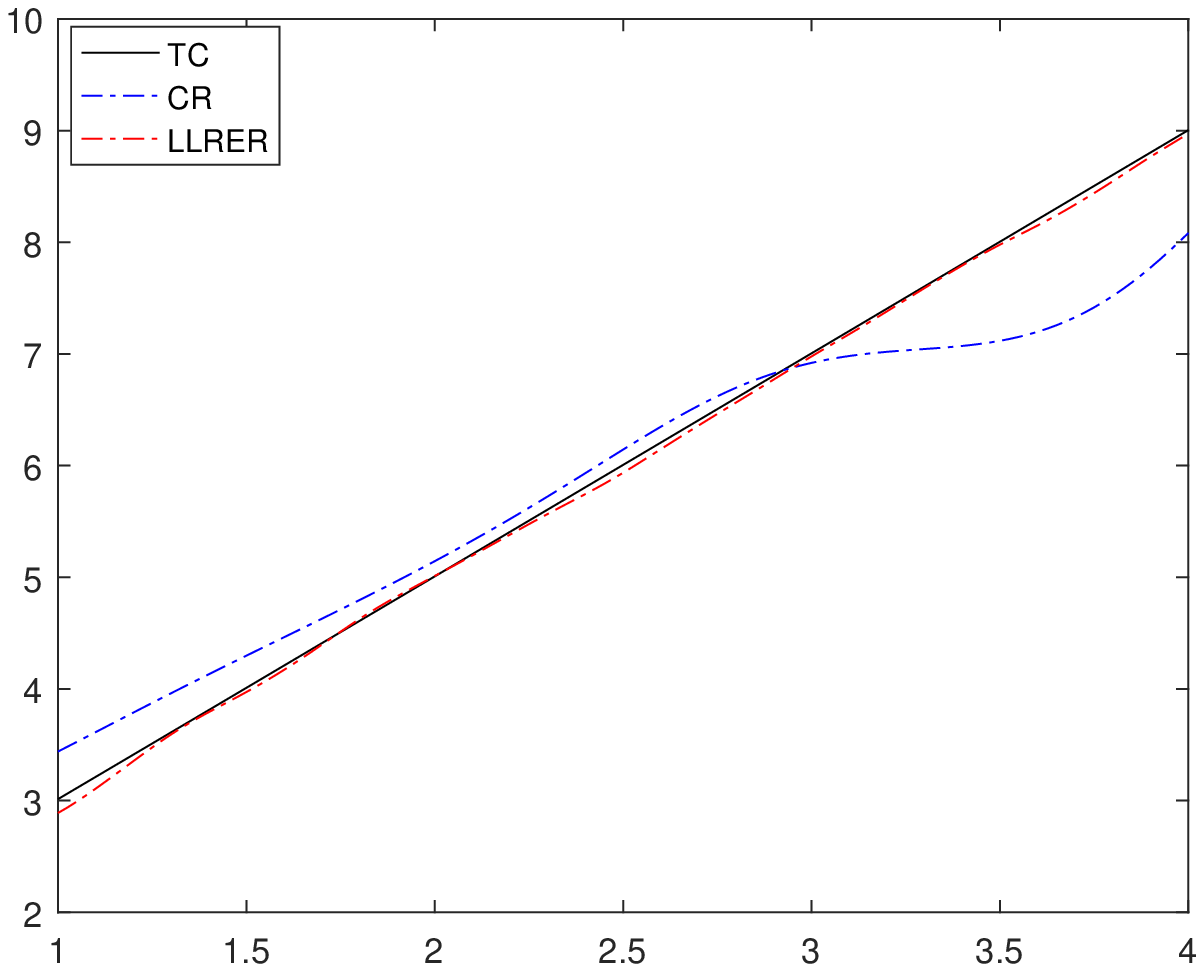}
	\end{minipage} \hfill
	\begin{minipage}[c]{.26\linewidth}
		\includegraphics[height=2in, width=2in]{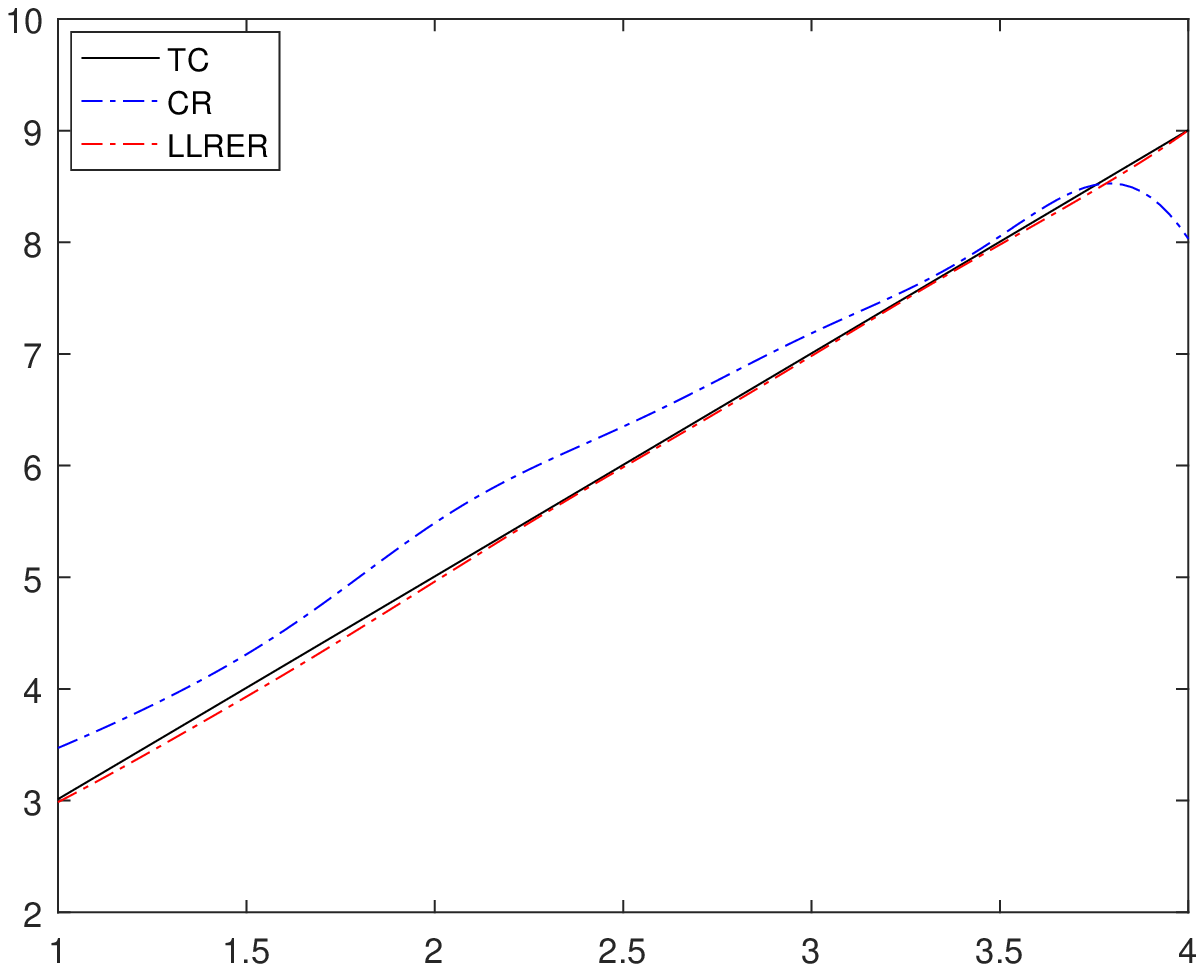}
	\end{minipage} \hfill
	\begin{minipage}[c]{.26\linewidth}
		\includegraphics[height=2in, width=2in]{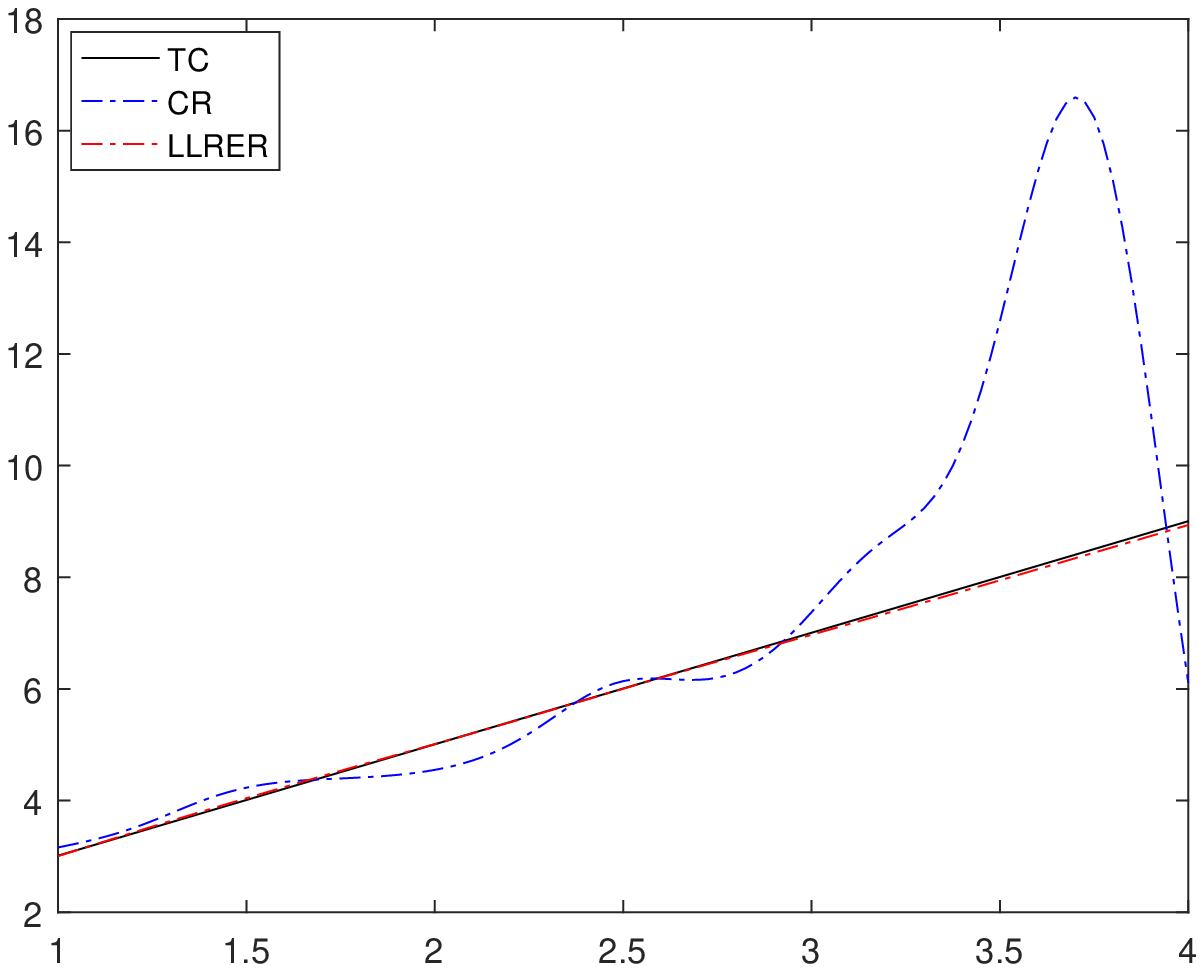}
	\end{minipage}\hfill\hfill
	\caption{\textcolor{black}{$\mu(\cdot)$}, \textcolor{red}{$\widehat{\mu}(\cdot)$} and \textcolor{blue}{$\widehat{m}(\cdot)$} with $n=300$ for C.P.$\approx 35,50$ and $70\%$ respectively.}\label{figure4}
\end{figure}
\vspace*{.51in}

\noindent {\it Effect of outliers:} \\
We compare the two models when the data contains outliers in the observed response value and we note that there is a significant difference between the two estimators for a fixed C.P. and sample size. As expected, when there are outliers, the relative regression estimator performs better than the Nadaraya-Watson and local linear estimators $m_n(\cdot)$ with respect to the number of outliers (see \hyperref[figure5]{Figure 5}). 
\begin{figure}[!h]
	\hspace*{-0.5cm}
	\begin{minipage}[c]{.26\linewidth}
		\includegraphics[height=2in, width=2in]{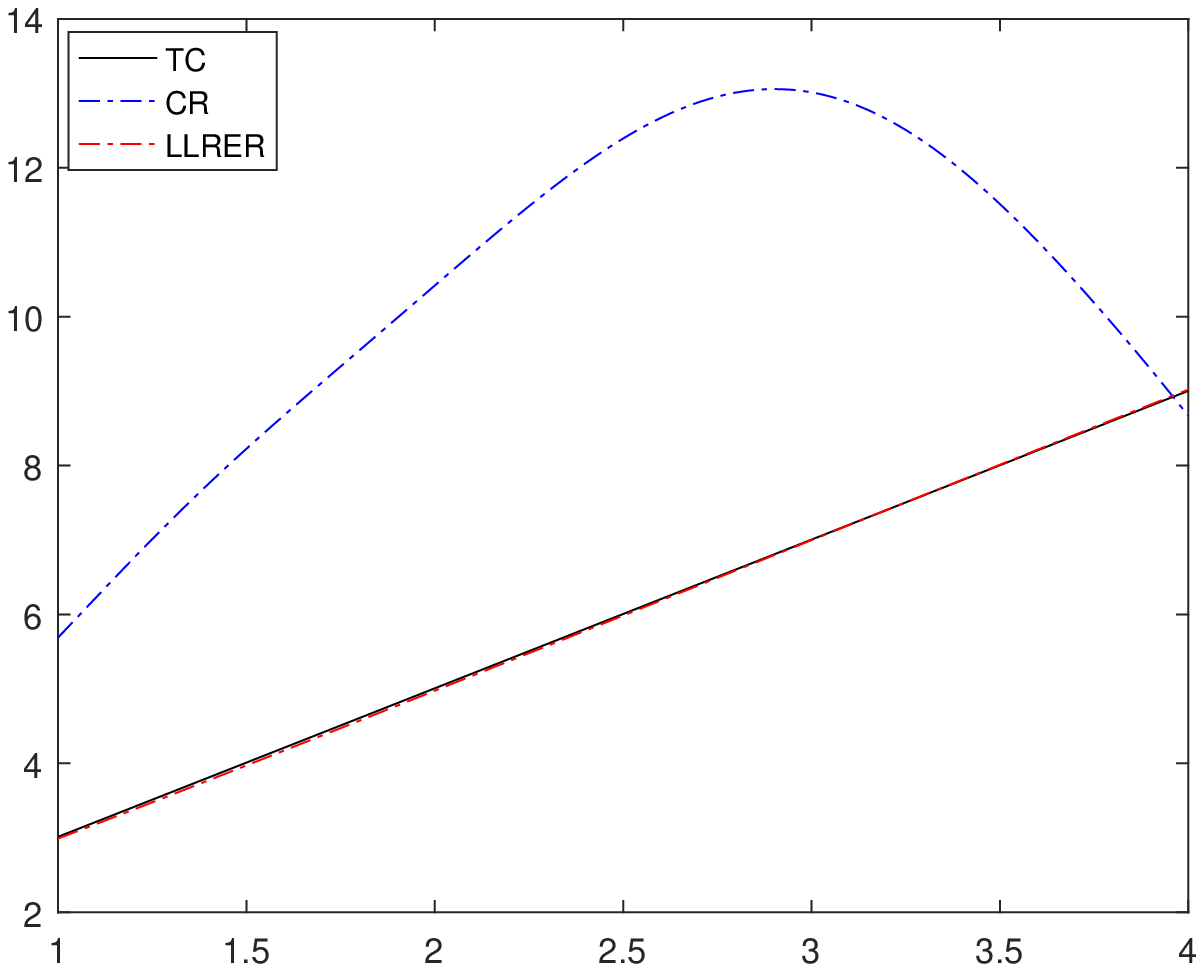}
	\end{minipage} \hfill
	\begin{minipage}[c]{.26\linewidth}
		\includegraphics[height=2in, width=2in]{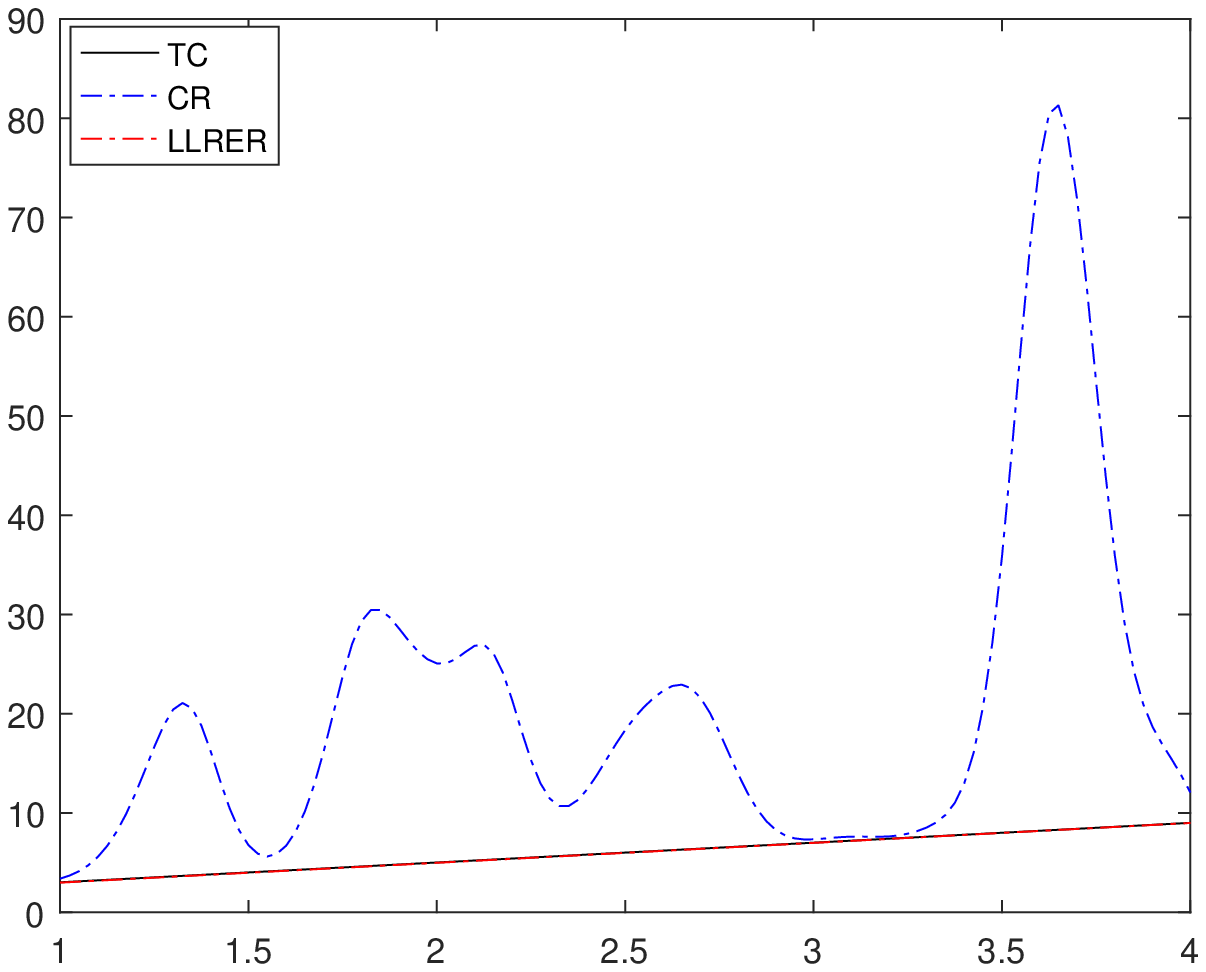}
	\end{minipage} \hfill
	\begin{minipage}[c]{.26\linewidth}
		\includegraphics[height=2in, width=2in]{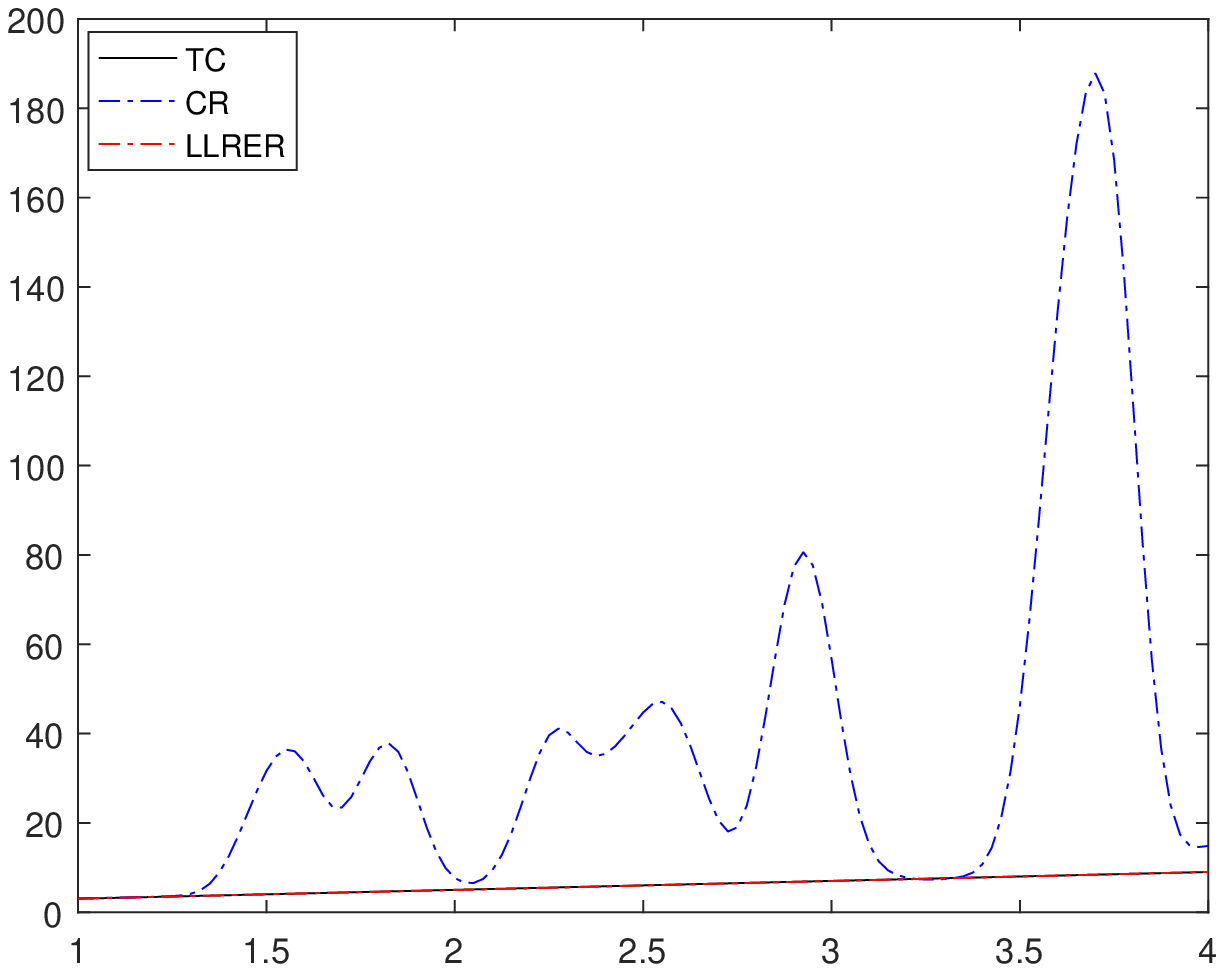}
	\end{minipage}\hfill\hfill
	\caption{\textcolor{black}{$\mu(\cdot)$}, \textcolor{red}{$\widehat{\mu}(\cdot)$} and \textcolor{blue}{$\widehat{m}(\cdot)$} with $n=300$ for C.P.$\approx 35\%$ and M.C.$=25,50,100$ respectively.}\label{figure5}
\end{figure}
\vspace*{.21in}

\subsubsection{LLCR vs LLRER}
\noindent  {\it Effect of C.P.:} \\
We observe from \hyperref[figure6]{Figure 6} that there is no meaningful difference between the LLCR and LLRER when the C.P. is low. The two predictors are basically equivalent and both show the good behavior. However for  high censorship rate our estimator remains resistant unlike its competitor which moves away from the edges.
\begin{figure}[!h]
	\hspace*{-0.5cm}
	\begin{minipage}[c]{.26\linewidth}
		\includegraphics[height=2in, width=2in]{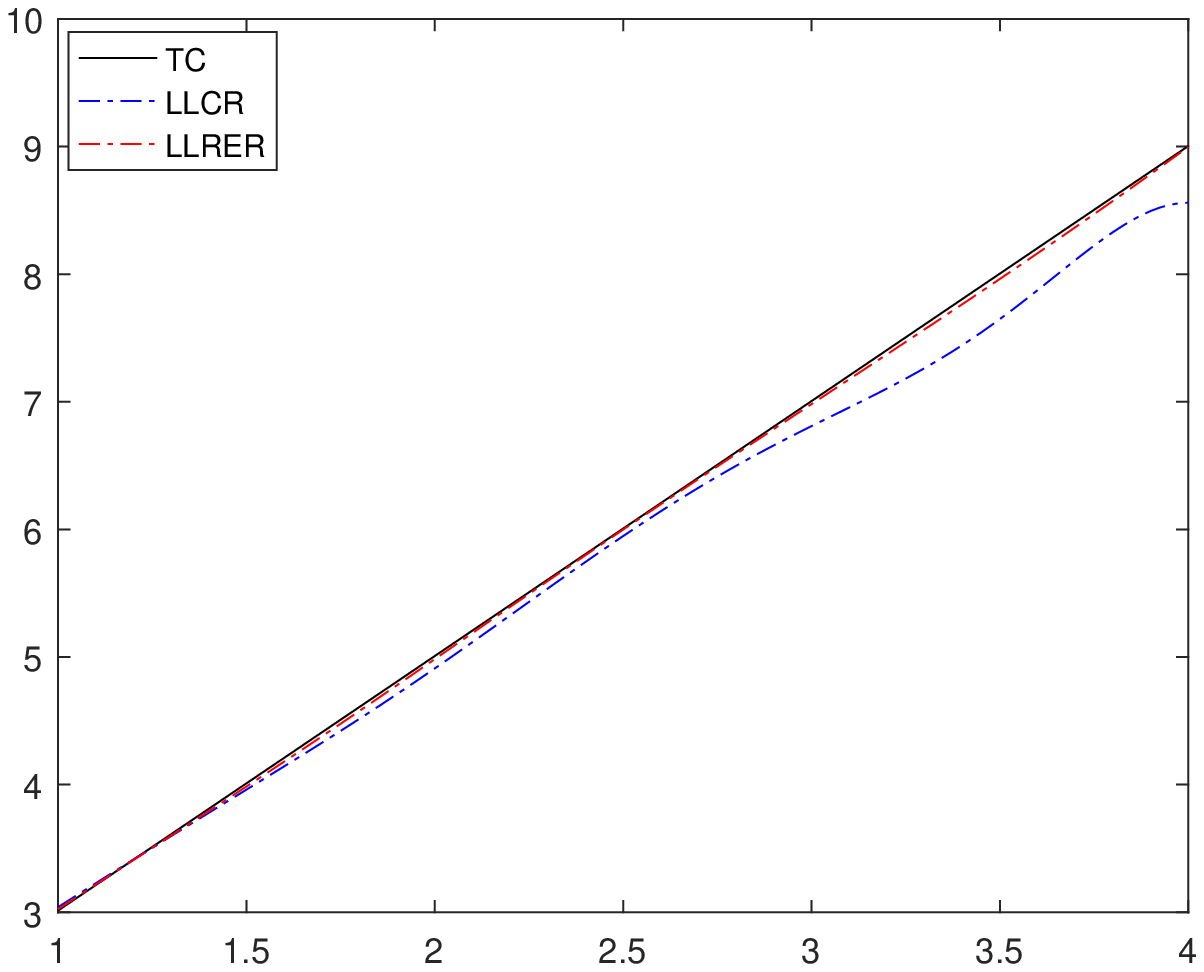}
	\end{minipage} \hfill
	\begin{minipage}[c]{.26\linewidth}
		\includegraphics[height=2in, width=2in]{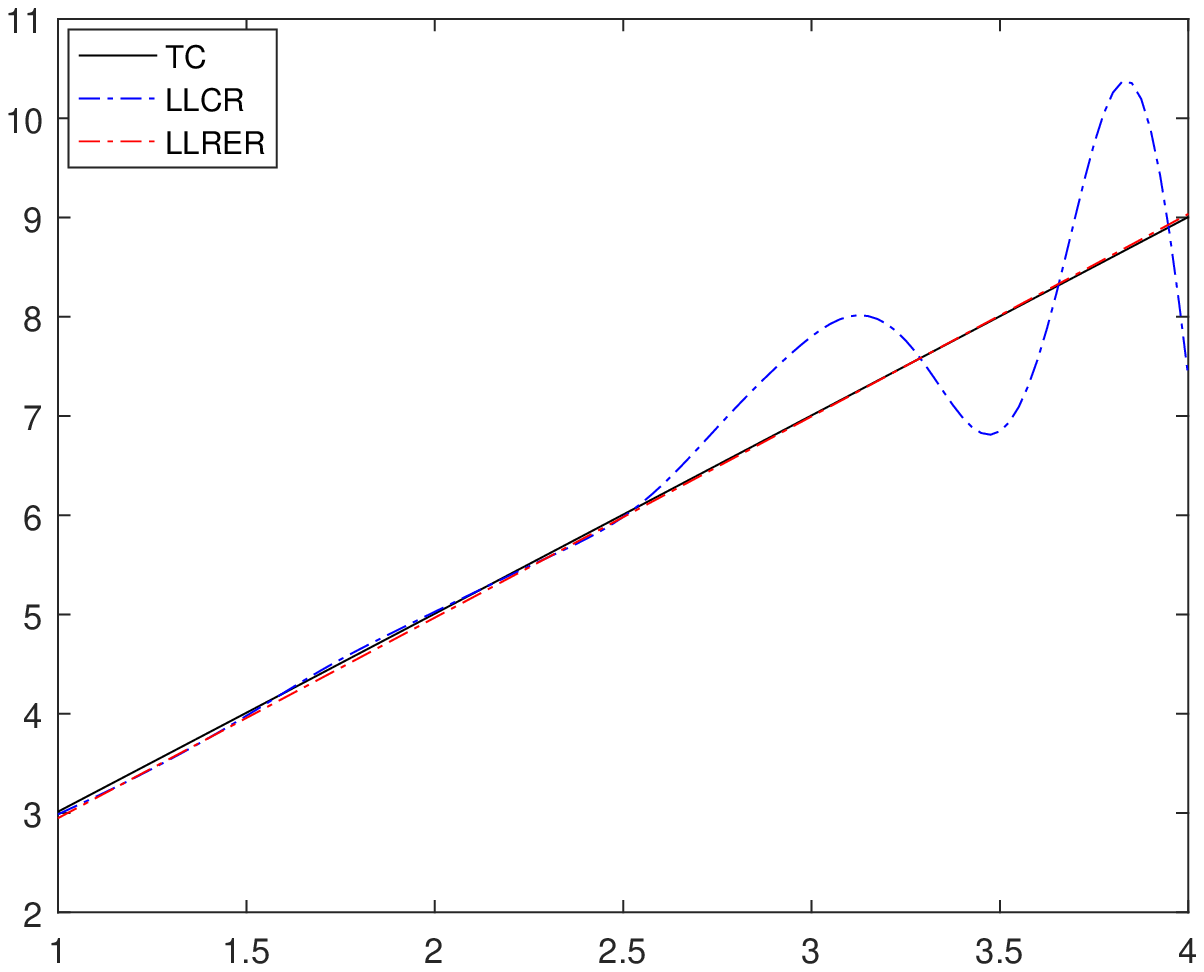}
	\end{minipage} \hfill
	\begin{minipage}[c]{.26\linewidth}
		\includegraphics[height=2in, width=2in]{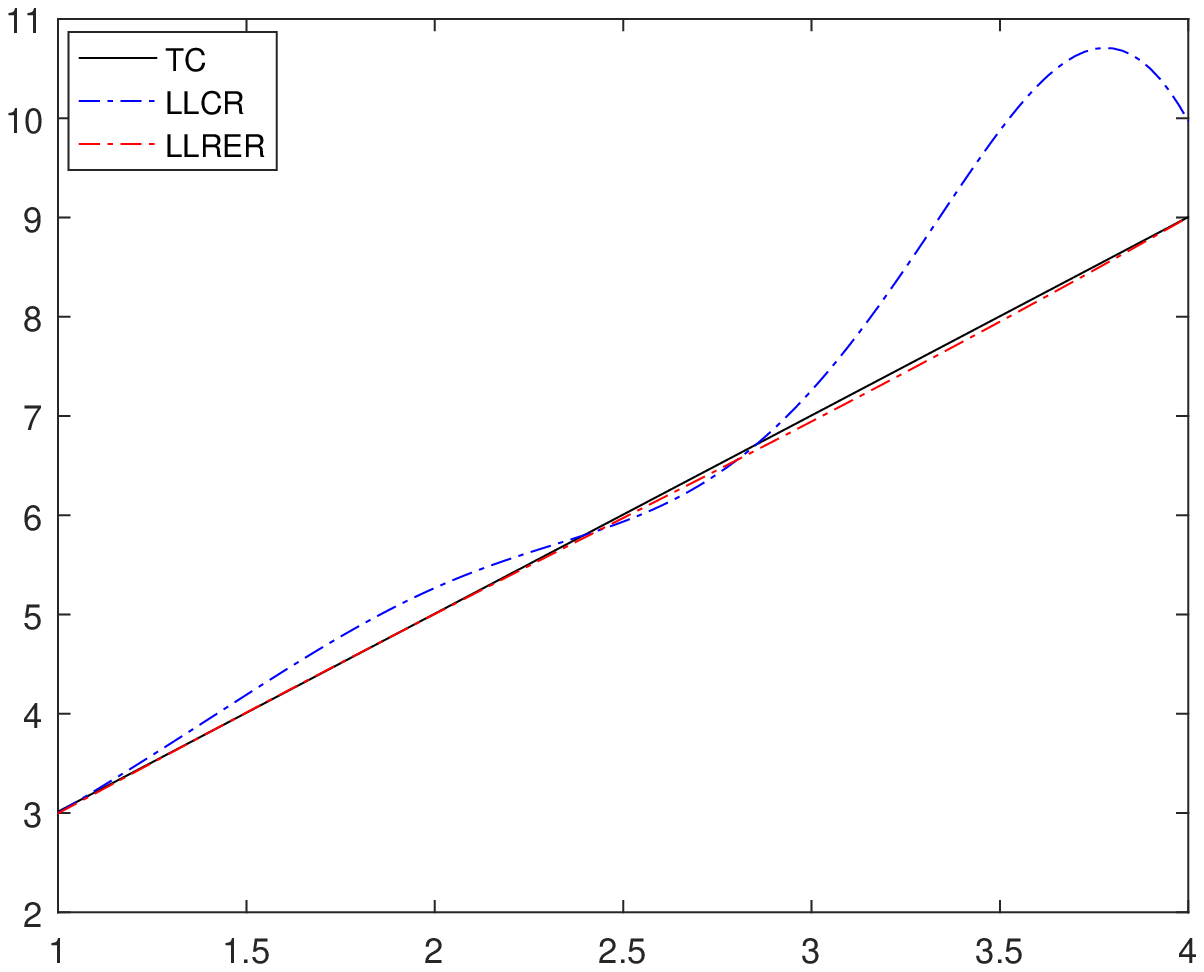}
	\end{minipage}\hfill\hfill
	\caption{\textcolor{black}{$\mu(\cdot)$}, \textcolor{red}{$\widehat{\mu}(\cdot)$} and \textcolor{blue}{$m_n(\cdot)$} with $n=300$ for C.P.$\approx 35,50$ and $66\%$ respectively.}\label{figure6}
\end{figure}
\vspace*{.251in}

\noindent {\it Effect of outliers:}\\
\hyperref[figure7]{Figure 7} shows clearly that the curve of the LLCR estimator is moves away from the TC when the M.C. increases which reflect the effectiveness of the procedure in presence of outliers.
\begin{figure}[!h]
	\hspace*{-0.5cm}
	\begin{minipage}[c]{.26\linewidth}
		\includegraphics[height=2in, width=2in]{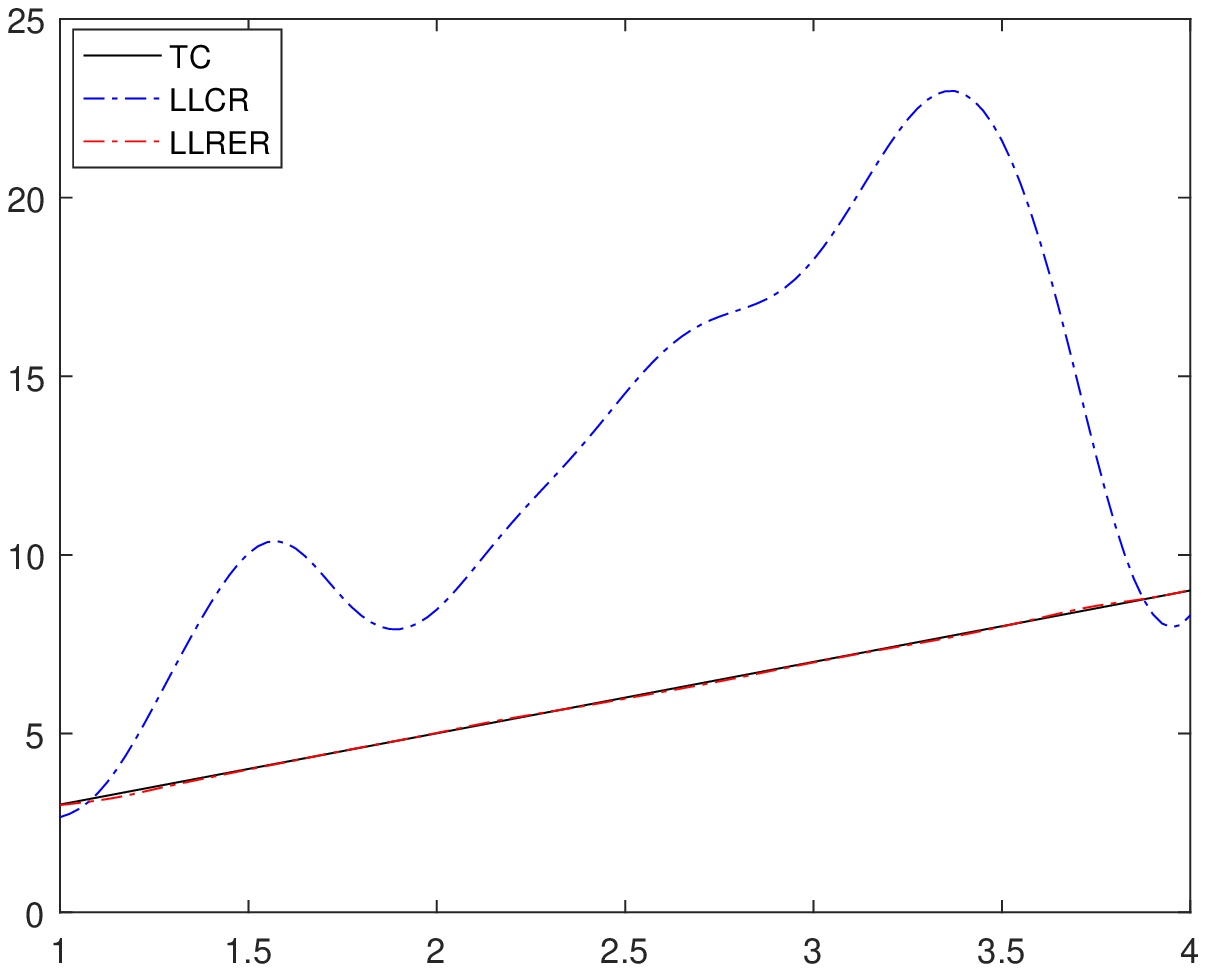}
	\end{minipage} \hfill
	\begin{minipage}[c]{.26\linewidth}
		\includegraphics[height=2in, width=2in]{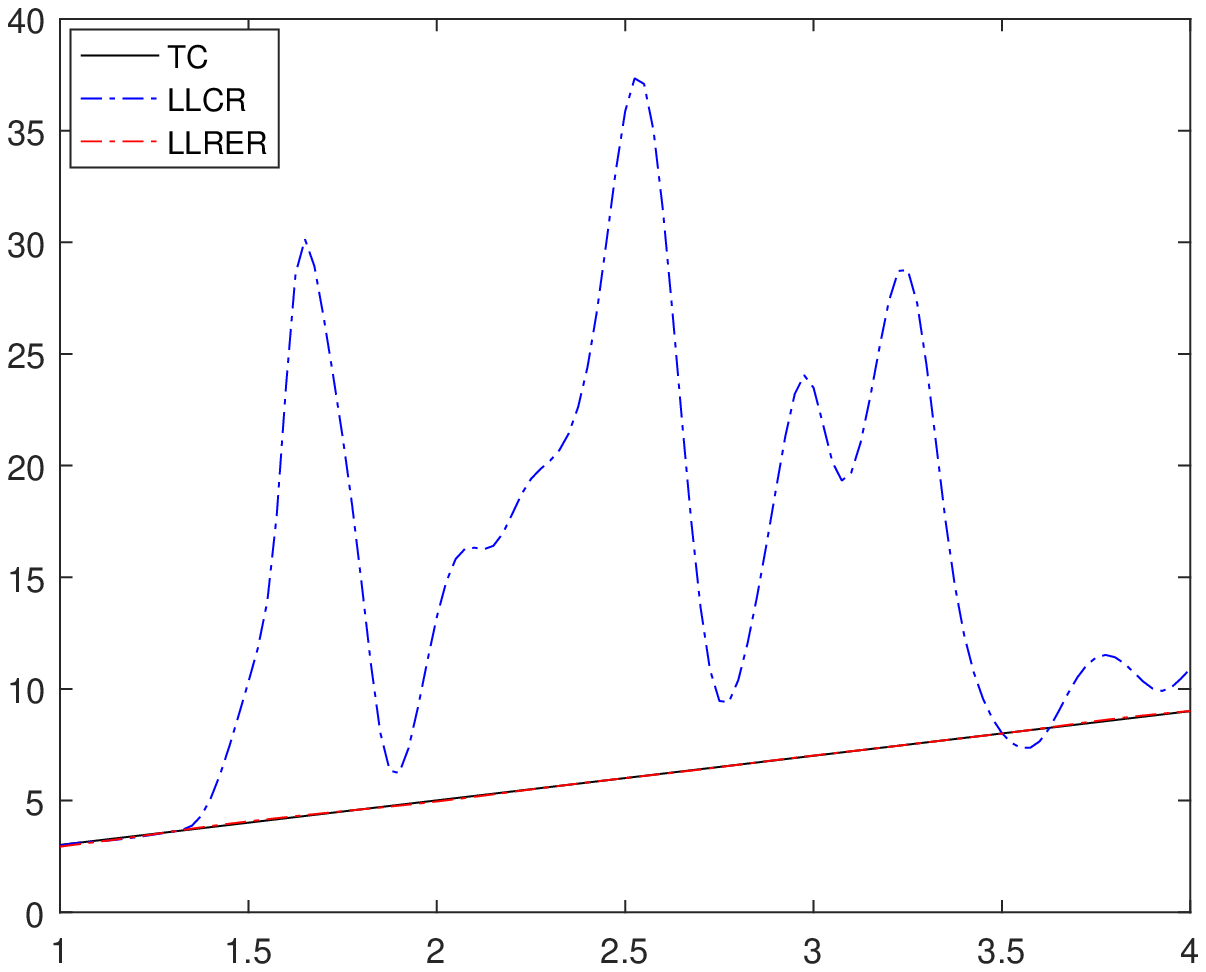}
	\end{minipage} \hfill
	\begin{minipage}[c]{.26\linewidth}
		\includegraphics[height=2in, width=2in]{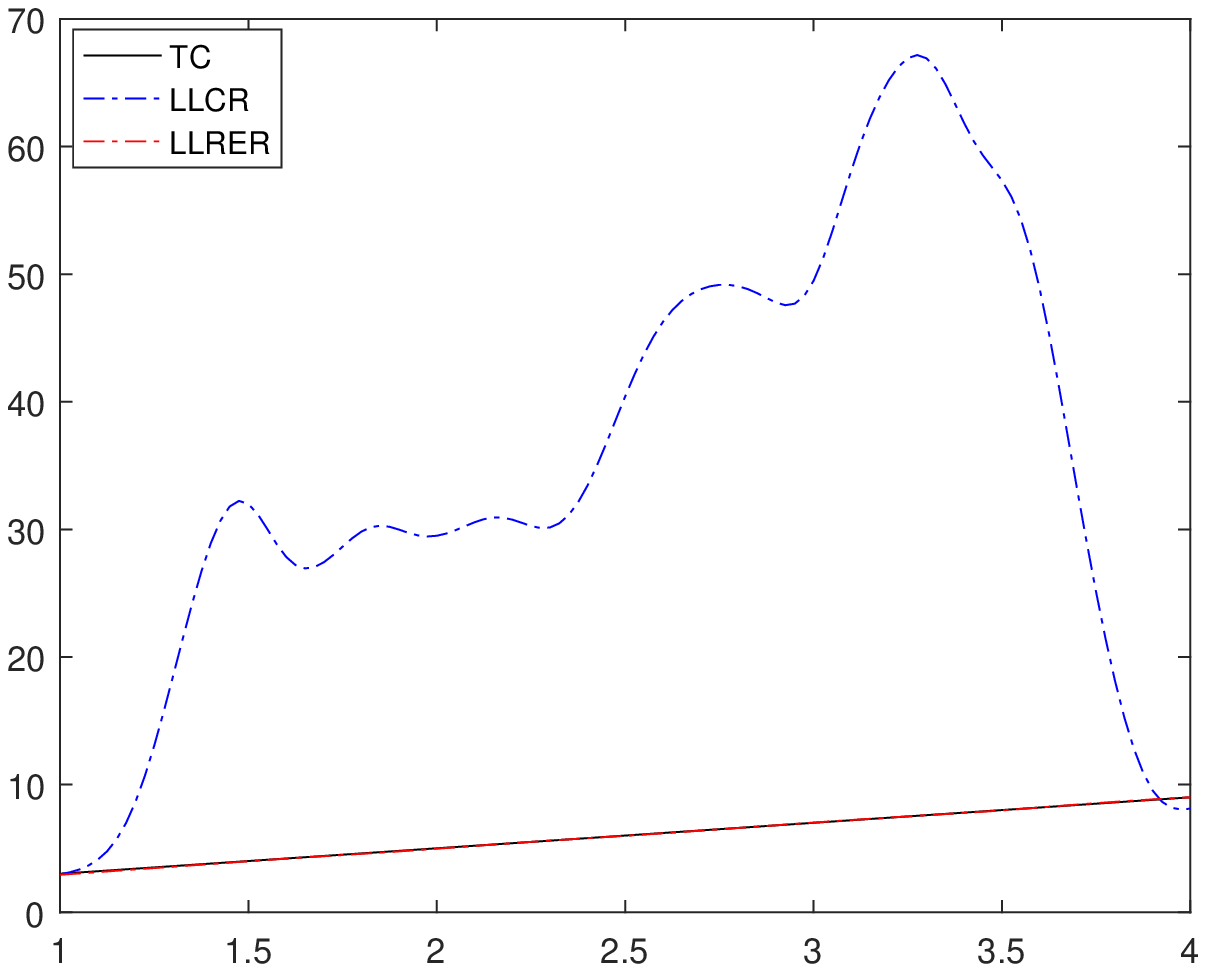}
	\end{minipage}\hfill\hfill
	\caption{\textcolor{black}{$\mu(\cdot)$}, \textcolor{red}{$\widehat{\mu}(\cdot)$} and \textcolor{blue}{$m_n(\cdot)$} with $n=300$ for C.P.$\approx 35\%$ and M.C.$=25,50,100$ respectively.}\label{figure7}
\end{figure}
\vspace*{0.3in}

\subsubsection{LLRER versus CR ans LLCR}
\begin{figure}[!h]
	\hspace*{-0.5cm}
	\begin{minipage}[c]{.26\linewidth}
		\includegraphics[height=2in, width=2in]{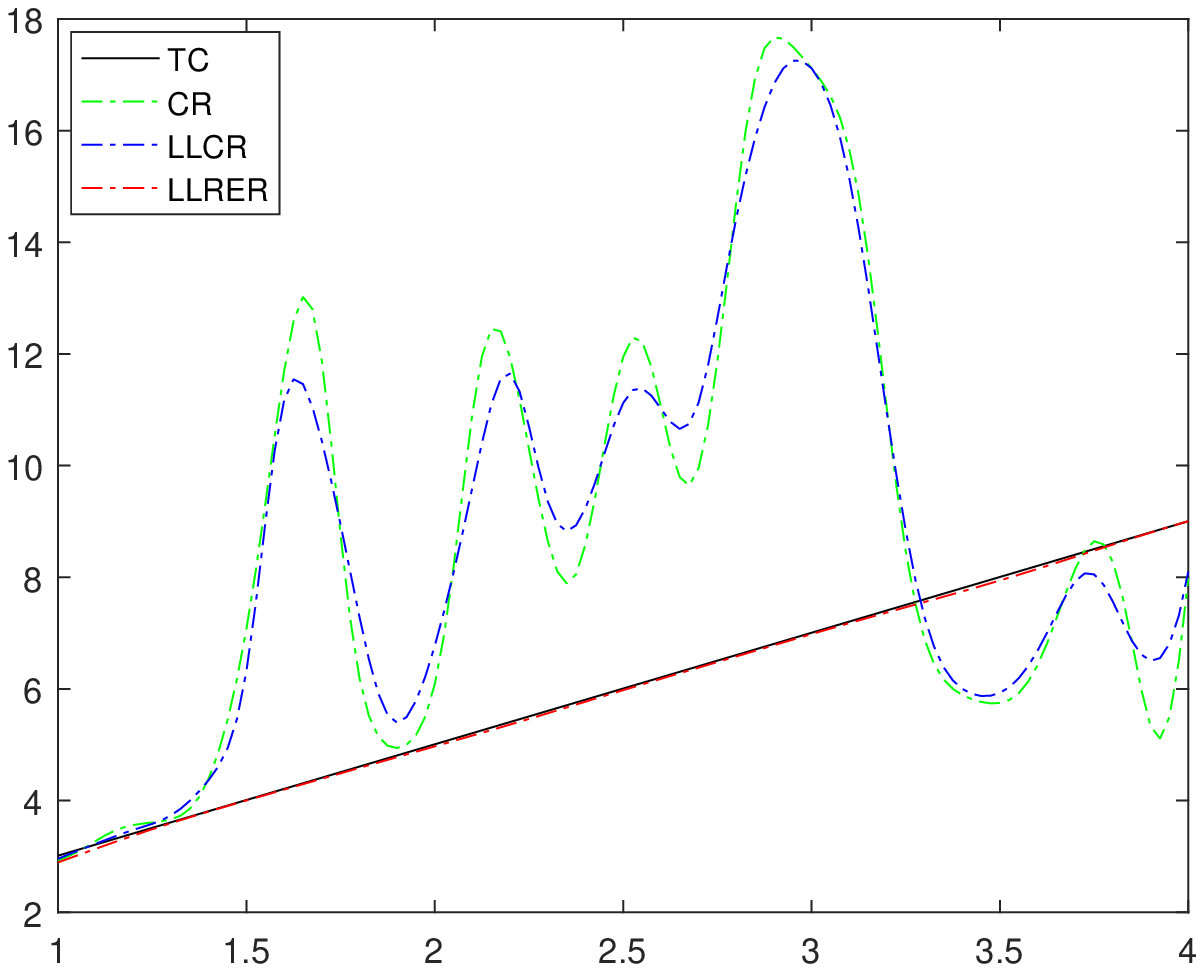}
	\end{minipage} \hfill
	\begin{minipage}[c]{.26\linewidth}
		\includegraphics[height=2in, width=2in]{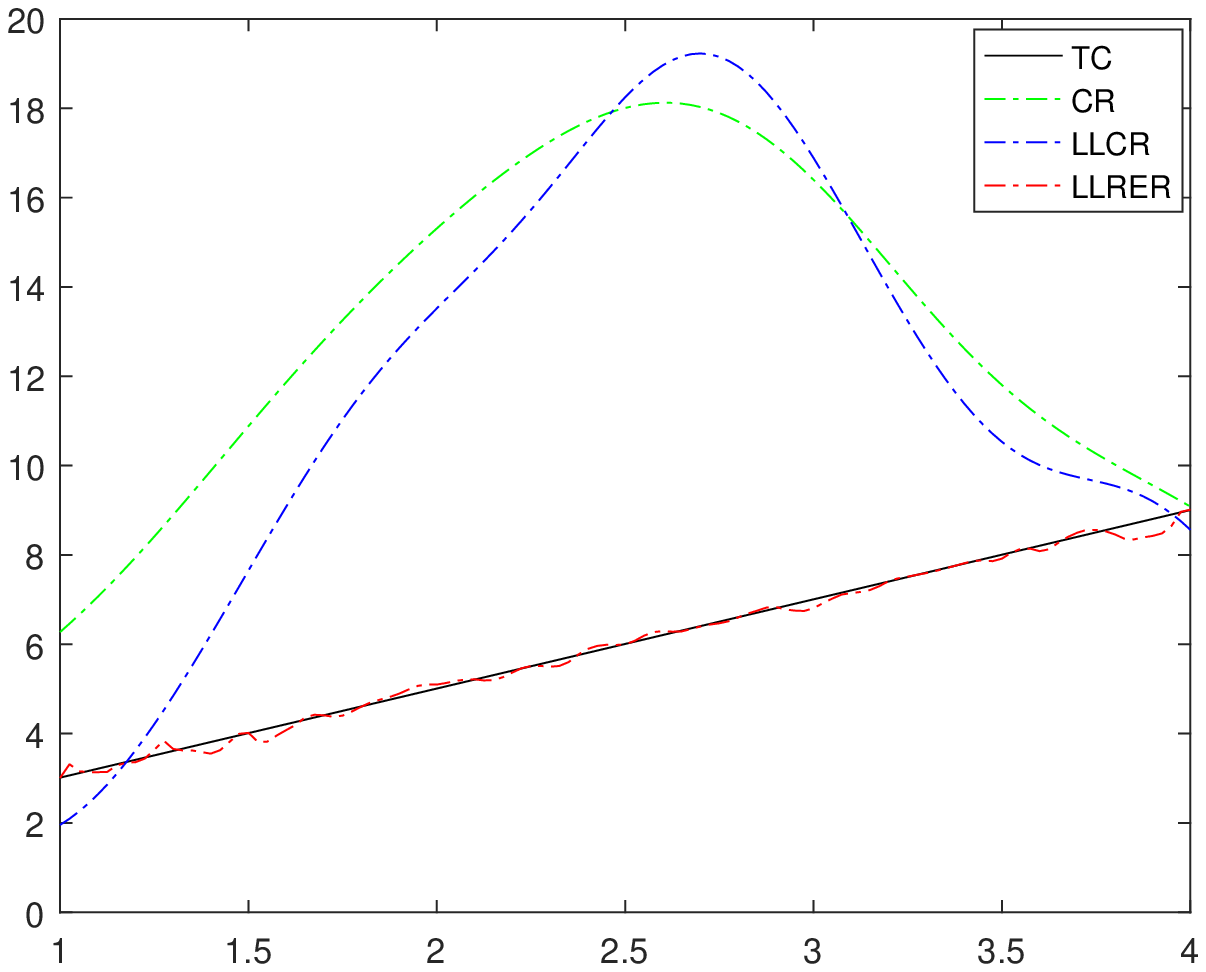}
	\end{minipage} \hfill
	\begin{minipage}[c]{.26\linewidth}
		\includegraphics[height=2in, width=2in]{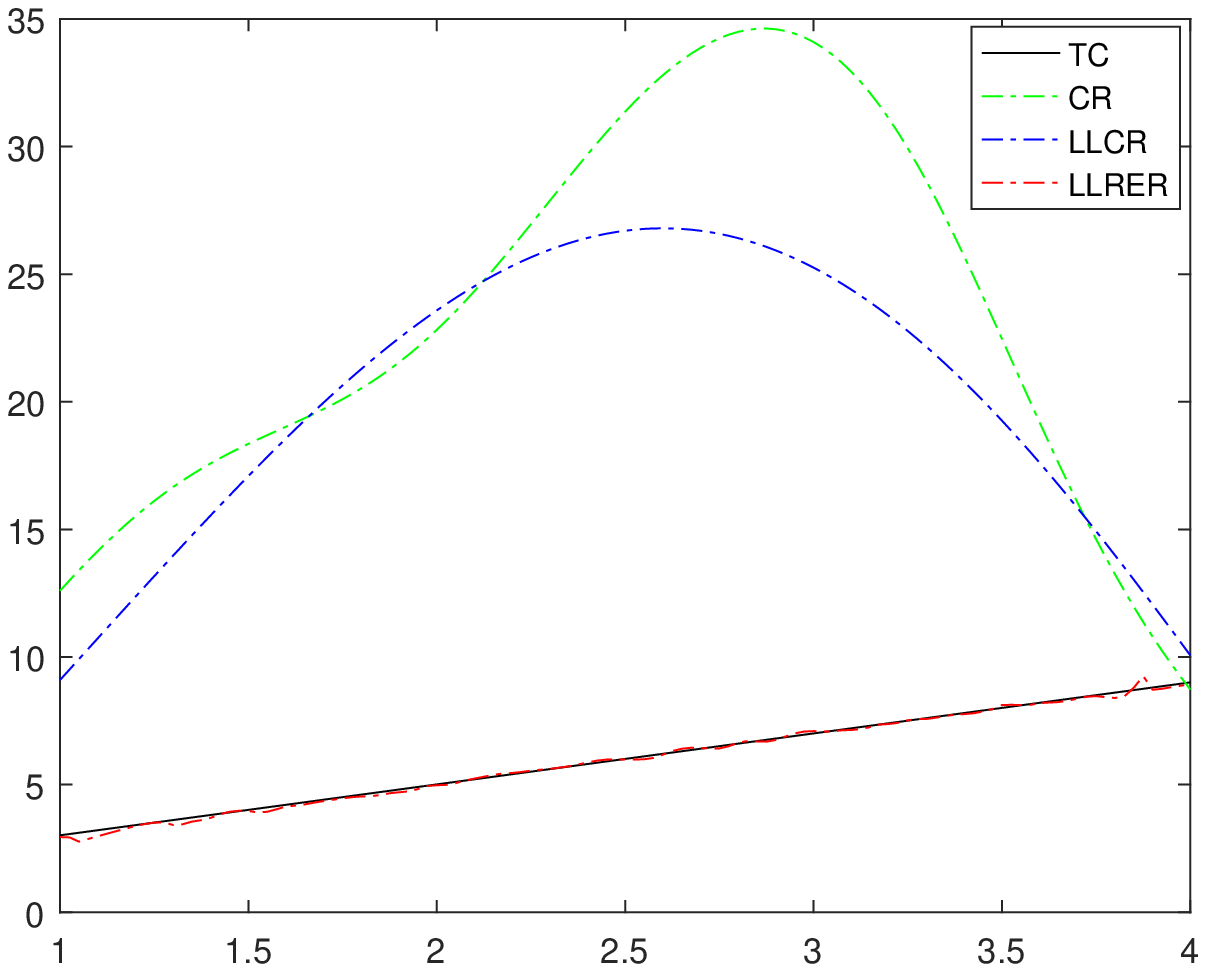}
	\end{minipage}\hfill\hfill
	\caption{\textcolor{black}{$\mu(\cdot)$}, \textcolor{green}{$\widehat{m}(\cdot)$}, \textcolor{red}{$\widehat{\mu}(\cdot)$} and \textcolor{blue}{$m_n(\cdot)$} with $n=300$ for C.P.$\approx 35\%$ and M.C.$=25,50,100$ respectively.}\label{figure8}
\end{figure}
\noindent Finally, in this figure, we can clearly see that in the presence of outliers, the new estimator obtained by combining the RER and LL methods is much more efficient compared to the two methods treated separately as that has been treated by many authors.

\section{Proofs and auxiliary results}\label{sect 5}
\begin{proof}[Proof of the Proposition 1.] Let introduce some notations for $\ell=1,2$ and $\gamma=0,1,2$:
	\begin{equation*}
	\widehat{S}_{\ell,\gamma}(x)=\frac{1}{nh} \sum_{i=1}^{n} \widehat{T}_i^{\star,-\ell} (X_i-x)^\gamma K_h(X_i-x) \;\;\;\text{and} \;\;\; 	\widetilde{S}_{\ell,\gamma}(x)=\frac{1}{nh} \sum_{i=1}^{n} T_i^{\star,-\ell} (X_i-x)^\gamma K_h(X_i-x).
	\end{equation*}
	We use the following decomposition:
	\begin{equation*}
	\begin{aligned}
	\widehat{\mu}_{\ell}(x)-\widetilde{\mu}_{\ell}(x)	&=\widehat{S}_{2,2}(x)\widehat{S}_{\ell,0}(x)-\widehat{S}_{2,1}(x)\widehat{S}_{\ell,1}(x)-\left(\widetilde{S}_{2,2}(x)\widetilde{S}_{\ell,0}(x)-\widetilde{S}_{2,1}(x)\widetilde{S}_{\ell,1}(x)\right)\\
	&=\widehat{S}_{2,2}(x)\widehat{S}_{\ell,0}(x)-\widetilde{S}_{2,2}(x)\widetilde{S}_{\ell,0}(x) - \left(\widehat{S}_{2,1}(x)\widehat{S}_{\ell,1}(x)-\widetilde{S}_{2,1}(x)\widetilde{S}_{\ell,1}(x)\right)\\
	&=:\mathcal{B}_{\ell,1}(x)-\mathcal{B}_{\ell,2}(x).
	\end{aligned}
	\end{equation*}
	On the one hand, for $\ell=1,2$, we get
	\begin{eqnarray}\label{decom1}
	\mathcal{B}_{\ell,1}(x)&=&\left( \widehat{S}_{2,2}(x)- \widetilde{S}_{2,2}(x)\right)\left(\widehat{S}_{\ell,0}(x)- \widetilde{S}_{\ell,0}(x)\right) +\left(\widetilde{S}_{\ell,0}(x) -\E[\widetilde{S}_{\ell,0}(x)]\right)\left( \widehat{S}_{2,2}(x)- \widetilde{S}_{2,2}(x)\right)\nonumber\\
	&+&\E[\widetilde{S}_{\ell,0}(x)]\left( \widehat{S}_{2,2}(x)- \widetilde{S}_{2,2}(x)\right)+ \left(\widetilde{S}_{2,2}(x)- \E[\widetilde{S}_{2,2}(x)]\right)\left(\widehat{S}_{\ell,0}(x)-\widetilde{S}_{\ell,0}(x)\right)\nonumber\\
	&+& \E[\widetilde{S}_{2,2}(x)]\left(\widehat{S}_{\ell,0}(x)-\widetilde{S}_{\ell,0}(x)\right).
	\end{eqnarray}
	On the other hand, for $\ell=1,2$, we get
	\begin{eqnarray}\label{decom2}
	\mathcal{B}_{\ell,2}(x)&=&\left( \widehat{S}_{2,1}(x)- \widetilde{S}_{2,1}(x)\right)\left(\widehat{S}_{\ell,1}(x)- \widetilde{S}_{\ell,1}(x)\right) +\left(\widetilde{S}_{2,1}(x) -\E[\widetilde{S}_{2,1}(x)]\right)\left( \widehat{S}_{\ell,1}(x)- \widetilde{S}_{\ell,1}(x)\right)\nonumber\\
	&+&\E[\widetilde{S}_{2,1}(x)]\left( \widehat{S}_{\ell,1}(x)- \widetilde{S}_{\ell,1}(x)\right)+ \left(\widetilde{S}_{\ell,1}(x)- \E[\widetilde{S}_{\ell,1}(x)]\right)\left(\widehat{S}_{2,1}(x)-\widetilde{S}_{2,1}(x)\right)\nonumber\\
	&+& \E[\widetilde{S}_{\ell,1}(x)]\left(\widehat{S}_{2,1}(x)-\widetilde{S}_{2,1}(x)\right).
	\end{eqnarray}
	It remains to study each term of the decomposition (\ref{decom1}) and (\ref{decom2}). For this, we will state and proof the following three  \hyperref[lem1]{Lemma 5.1-5.3}.

	\begin{lemma}\label{lem1}
		Under hypotheses \hyperref[H2i]{H2 i)} and \hyperref[H3]{H3}, for $\ell=1,2$, $\gamma=0,1,2$, and $n$ large enough, we have 
		\begin{equation*}
		\sup_{x \in \mathcal{C}}\Big| \widehat{S}_{\ell,\gamma}(x)-\widetilde{S}_{\ell,\gamma}(x)\Big|=\text{O}_{a.s.} \left(\sqrt{\frac{\log\log n}{n}}\right).
		\end{equation*}
	\end{lemma}
	\begin{proof}[Proof of Lemma 5.1.] For $\ell=1,2$, $\gamma=0,1,2$, we have 
		\begin{equation*}
		\begin{aligned}
		\sup_{x \in \mathcal{C}}\big| \widehat{S}_{\ell,\gamma}(x)-\widetilde{S}_{\ell,\gamma}(x)\big|&= \sup_{x \in \mathcal{C}} \Big| \frac{1}{nh} \sum_{i=1}^{n} \widehat{T}_i^{\star,-\ell}(X_i-x)^\gamma K_h(X_i-x)-\frac{1}{nh} \sum_{i=1}^{n} T_i^{\star,-\ell}(X_i-x)^\gamma K_h(X_i-x)\Big|\\
		&= \sup_{x \in \mathcal{C}}\Big| \frac{1}{nh} \left(\sum_{i=1}^{n} \frac{\delta_i Y^{-\ell}_i}{\overline{G}_n(Y_i)}(X_i-x)^\gamma K_h(X_i-x)- \sum_{i=1}^{n} \frac{\delta_i Y^{-\ell}_i}{\overline{G}(Y_i)}(X_i-x)^\gamma K_h(X_i-x)\right)\Big|\\
		&= \sup_{x \in \mathcal{C}} \Big| \frac{1}{nh} \sum_{i=1}^{n} \delta_i T^{-\ell}_i(X_i-x)^\gamma K_h(X_i-x) \left( \frac{1}{\overline{G}_n(Y_i)}-\frac{1}{\overline{G}(Y_i)}\right)\Big|\\
		& \leq \frac{1}{\overline{G}^2(\tau_F)} \sup_{t \leq \tau_F} \big| \overline{G}_n(t)-\overline{G}(t)\big| \times \sup_{x \in \mathcal{C}} \Big| \frac{1}{nh} \sum_{i=1}^{n} T^{-\ell}_i(X_i-x)^\gamma K_h(X_i-x) \Big|\\
		&=: \sup_{t \leq \tau_F} \mathcal{D}_1(t) \times \sup_{x \in \mathcal{C}} \big| \mathcal{D}_2(x)\big|.
		\end{aligned}
		\end{equation*}
		From Lemma 4.2. in \cite{Deuheuvels2000}, the first term of the right hand side is equal to:
		\begin{equation}\label{d1}
		\sup_{t \leq \tau_F} \mathcal{D}_1(t)=\text{O}_{a.s.}\left(\sqrt{\frac{\log \log n}{n}}\right) \quad \quad \quad \text{as} \quad \quad \quad n \rightarrow \infty.
		\end{equation}
		For the second term and using  the strong law of large numbers we have 
		\[\sup_{x \in \mathcal{C}}\left|\mathcal{D}_2(x)\right| \leq 	C \sup_{x \in \mathcal{C}} \left|\E\left[h^{-1} (X_1-x)^\gamma K_h(X_1-x)\right] \right|.\]
		By a change of variable, Taylor expansion and with the condition (\ref{bounded}), we get
		\begin{equation*}
		\begin{aligned}
		\E\left[h^{-1} (X_1-x)^\gamma K_h(X_1-x)\right]&= h^{-1} \int (u-x)^\gamma K_h(u-x) f(u) du\\
		&=h^{-1} \int (vh)^\gamma K(v) f(x+vh) hdv\\
		&= h^\gamma f(x) \int v^\gamma K(v) dv +h^{\gamma+1} \int v^{\gamma+1} K(v) f^{\prime}(\xi)) dv.\\
		\end{aligned}
		\end{equation*}
		Under the kernel hypothesis \hyperref[H2i]{ H2 i)} and the regularity hypothesis \hyperref[H3]{H3}, we get
		\begin{equation}\label{d2}
		\sup_{x \in \mathcal{C}}\left|\mathcal{D}_2(x)\right| =O\Big(h^{\gamma} \Big).
		\end{equation}
		Combining the results (\ref{d1}) and (\ref{d2}), the proof  of \hyperref[lem1]{Lemma 5.1} is achieved.
	\end{proof}
	
	\vspace*{-.25in}
	
	\begin{lemma}\label{lem2}
		Under hypotheses \hyperref[H1]{H1}, \hyperref[H2i]{H2 i)}, \hyperref[H3]{H3} and \hyperref[H4]{H4} for $\ell=1,2$, $\gamma=0,1,2,$ and $n$ large enough, we have 
		\begin{equation*}
		\sup_{x \in \mathcal{C}}\big| \widetilde{S}_{\ell,\gamma}(x)-\E[\widetilde{S}_{\ell,\gamma}(x)]\big|=\text{O}_{a.s.}\left(\sqrt{\frac{\log n}{nh}}\right).
		\end{equation*}
	\end{lemma}
	\begin{proof}[Proof of Lemma 5.2.] Let
		\begin{equation*}
		\mathcal{C}_n=\left\{ x_i-b_n,x_i+b_n,\;  1 \leq i \leq d_n\right\}
		\end{equation*}
		is the intervals extremities grid where $b_n=n^{-1/2q}$ for $q>0$ and cover the compact set $\mathcal{C}$ by $\displaystyle \cup_{i=1}^{d_n}\big[ x_i-b_n,x_i+b_n\big]$ with $d_n=\text{O}\left(n^{1/2q}\right)$. 
		\[\sup_{x \in \mathcal{C}} \big|\widetilde{S}_{\ell,\gamma}(x)-\E[\widetilde{S}_{\ell,\gamma}(x)]\big| \leq \max_{1\leq i \leq d_n} \max_{x \in \mathcal{C}_n} \big|\widetilde{S}_{\ell,\gamma}(x)-\E[\widetilde{S}_{\ell,\gamma}(x)]\big| + 2^q C b_n^q.\]
		using  $b_n=n^{-1/2q}$ then  \[b_n^q=\text{O}\left(\sqrt{\frac{\log n}{nh}}\right).\] 
		For this, observe that for all $\varepsilon>0$, 
		\begin{equation*}
		\mathbb{P}\left(\max_{x\in \mathcal{C}_n}\big|\widetilde{S}_{\ell,\gamma}(x)-\E[\widetilde{S}_{\ell,\gamma}(x)]\big|>\varepsilon \right) \leq \sum_{x \in \mathcal{C}_n} \mathbb{P} \left(\big|\widetilde{S}_{\ell,\gamma}(x)-\E[\widetilde{S}_{\ell,\gamma}(x)]\big|>\varepsilon\right).
		\end{equation*}
		Let us write for $\ell=1,2$, $\gamma=0,1,2$ and $x \in \mathcal{C}_n$
		\begin{equation*}
		\begin{aligned}
		\widetilde{S}_{\ell,\gamma}(x)-\E[\widetilde{S}_{\ell,\gamma}(x)]&=\frac{1}{nh} \sum_{i=1}^{n}T^{\star,-\ell}_i(X_i-x)^{\gamma} K_h(X_i-x) - \E\left[\frac{1}{nh} \sum_{i=1}^{n}T^{\star,-\ell}_i(X_i-x)^{\gamma} K_h(X_i-x)\right]\\
		&=\frac{1}{n} \displaystyle \sum_{i=1}^{n}\frac{T^{\star,-\ell}_i(X_i-x)^{\gamma} K_h(X_i-x) - \E\left[T^{\star,-\ell}_i(X_i-x)^{\gamma} K_h(X_i-x)\right]}{h}\\
		&=:\frac{1}{n} \sum_{i=1}^{n} \mathcal{A}_{\gamma,i}^{\ell}(x).
		\end{aligned}
		\end{equation*}
		In view of \hyperref[coro]{Corollary A.8.} (see Appendix), we focus on the absolute moments of order $\nu$ of $\mathcal{A}_{\gamma,i}^{\ell}(x)$
		\begin{equation*}
		\begin{aligned}
		\E|\mathcal{A}_{\gamma,i}^{\ell}(x)|^\nu &= \E \left| h^{-\nu} \left(T^{\star,-\ell}_i(X_i-x)^{\gamma} K_h(X_i-x)- \E\left[T^{\star,-\ell}_i(X_i-x)^{\gamma} K_h(X_i-x)\right]\right)^\nu\right|\\
		&= h^{-\nu} \E \left| \sum_{k=0}^{\nu} c_{k,\nu} \left(T^{\star,-\ell}_i(X_i-x)^\gamma K_h(X_i-x)\right)^k \E\left[T^{\star,-\ell}_i(X_i-x)^\gamma K_h(X_i-x)\right]^{\nu-k}\right|\\
		&\leq h^{-\nu} \sum_{k=0}^{\nu} c_{k,\nu}  \left| \E\left[\left(T^{\star,-\ell}_1(X_1-x)^\gamma K_h(X_1-x)\right)^k\right] \E\left[T^{\star,-\ell}_1(X_1-x)^\gamma K_h(X_1-x)\right]^{\nu-k}\right|.\\
		\end{aligned}
		\end{equation*}
		On the one hand, using the conditional expectation property, Taylor expansion and  under \hyperref[H2i]{H2 i)} and \hyperref[H5]{H5}, we have 
		\begin{equation*}
		\begin{aligned}
		\E\left[\left(T^{\star,-\ell}_1(X_1-x)^\gamma K_h(X_1-x)\right)^k\right]&=\E\left[T^{\star,-k\ell}_1(X_1-x)^{\gamma k} K_h^k(X_1-x)\right]\\
		&=\E\left[(X_1-x)^{\gamma k} K_h^k(X_1-x)\E[T^{\star,-k\ell}_1|X_1]\right]\\
		&=\int (u-x)^{\gamma k} K_h^k(u-x)\E[T^{\star,-k\ell}_1|X_1=u]f(u)du\\
		\end{aligned}
		\end{equation*}  
		\begin{equation*}
		\begin{aligned}
		&\hspace*{1cm}\leq \frac{1}{\overline{G}^{k-1}(\tau_F)}\int (u-x)^{\gamma k} K_h^k(u-x)\int t^{-\ell k}f_{T|X}(t|u)dt f(u)du\\
		&\hspace*{1cm}= \frac{1}{\overline{G}^{k-1}(\tau_F)}\int (u-x)^{\gamma k} K_h^k(u-x)\upsilon_{\ell,k}(u)du\\
		&\hspace*{1cm}= \frac{h^{\gamma k+1}}{\overline{G}^{k-1}(\tau_F)}\int s^{\gamma k} K^k(s)\upsilon_{\ell,k}(x+sh)ds\\
		&\hspace*{1cm}= \frac{h^{\gamma k+1}}{\overline{G}^{k-1}(\tau_F)}\upsilon_{\ell,k}(x) \int s^{\gamma k} K^k(s)ds+\frac{h^{\gamma k+2}}{\overline{G}^{k-1}(\tau_F)}\int s^{\gamma k+1} K^k(s)       \upsilon_{\ell,k}^{\prime}(\xi)ds.
		\end{aligned}
		\end{equation*}  
		On the other hand, using the same arguments as previously and under \hyperref[H2i]{H2 i)}, \hyperref[H4]{H4} we have
		\begin{equation*}
		\begin{aligned}
		\E\left[T_1^{\star,-\ell}(X_1-x)^{\gamma} K_h(X_1-x)\right]^{\nu-k}&=\left(\int (u-x)^{\gamma} K_h(X_1-x)\E\left[T_1^{\star,-\ell}|X_1=u\right] f(u)du\right)^{\nu-k}\\
		&=\left(\int (hv)^{\gamma} K(v)\mu_{\ell}(u)f(u)du\right)^{\nu-k}\\
		&=\left(h^{\gamma+1}\int v^{\gamma} K(v)r_{\ell}(x+vh)dv\right)^{\nu-k}\\
		&\hspace*{-.61in}=\left(h^{\gamma+1}r_{\ell}(x)\int v^{\gamma} K(v)dv+h^{\gamma+2}\int v^{\gamma+1} K(v)r_{\ell}^{\prime}(\xi)dv\right)^{\nu-k}.\\
		\end{aligned}
		\end{equation*}
		Then, for $\ell=1,2$, $\gamma=0,1,2$ and for all $\nu \geq 2$, we get easily 
		\begin{equation*}
		\begin{aligned}
		\E|\mathcal{A}_{\gamma,1}^\ell(x)|^\nu &\leq \text{O}(h^{-\nu}) \times \text{O}(h^{\gamma k+1}) \times \text{O}(h^{(\gamma+1)(\nu-k)})\\
		&= \text{O}(h^{\gamma \nu-k+1})\\
		&= \text{O}(\max_{1 \leq k \leq \nu}h^{-k+1})\\
		&= \text{O}(h^{-\nu+1}).\\
		\end{aligned}
		\end{equation*}
		Now, we can apply the exponential inequality in \hyperref[coro]{Corollary A.8.} by choosing $a^2=h^{-1}$, we get
		\begin{equation*}
		\mathbb{P}\left( \left| \widetilde{S}_{\ell,\gamma}(x)-\E[\widetilde{S}_{\ell,\gamma}(x)] \right|>\varepsilon \right)=\mathbb{P}\left( \left|\sum_{i=1}^n \mathcal{A}_{\gamma,i}^\ell(x) \right|>\varepsilon n\right)\leq 2 \exp \left( -\frac{\varepsilon^2 nh}{2(1+\varepsilon)}\right).
		\end{equation*}
		Hence, for a fixed $\varepsilon_0$,  choosing $\varepsilon=\varepsilon_0 \left(\frac{\log n}{nh}\right)^{1/2}$,  we get
		\begin{equation*}
		\mathbb{P}\left( \left| \widetilde{S}_{\ell,\gamma}(x)-\E[\widetilde{S}_{\ell,\gamma}(x)] \right|>\varepsilon\right) \leq 2 \exp \left\{ -\frac{\varepsilon_0^2 \log n}{2\left(1+\varepsilon_0 \sqrt{\frac{\log n}{nh}}\right)}\right\}
		\end{equation*}
		and for $n$ large enough, we have  
		\begin{equation*}
		\mathbb{P}\left( \left| \widetilde{S}_{\ell,\gamma}(x)-\E[\widetilde{S}_{\ell,\gamma}(x)] \right|>\varepsilon\right) \leq 2 \exp \left(-\frac{\varepsilon_0^2}{4} \log n\right)=2 n^{-\frac{\varepsilon_0^2}{4}}
		\end{equation*}
		which gives 
		\begin{equation*}
		\sum_{x \in \mathcal{C}_n} \mathbb{P}\left(\left|\widetilde{S}_{\ell,\gamma}(x)-\E[\widetilde{S}_{\ell,\gamma}(x)]\right|>\varepsilon\right) \leq 4d_n n ^{-\frac{\varepsilon_0^2}{4}+\frac{\nu}{2}}.
		\end{equation*}
		Finally, an appropriate choice of $\varepsilon_0$ yields to an upper bound of order $n^{-3/2}$ and  by Borel-Cantelli's lemma we get the result.
	\end{proof}
	\begin{lemma}\label{lem3}
		Under Hypotheses \hyperref[H1]{H1}, \hyperref[H2i]{H2 i)} and \hyperref[H4]{H4}, for $\ell=1,2$ and $\gamma=0,1,2$, we have 
		\begin{equation*}
		\sup_{x \in \mathcal{C}} \left|\E[\widehat{S}_{\ell,\gamma}(x)]\right| =\text{O}\left(h^\gamma\right).
		\end{equation*}
	\end{lemma}
	\begin{proof}[Proof of Lemma 5.3.] 	
		Using the conditional expectation property, Taylor expansion, under Hypotheses \hyperref[H1]{H1}, \hyperref[H2i]{H2 i)} and \hyperref[H4]{H4} and using the fact that  $\E[T_1^{\star,-\ell}|X_1=u]=\mu_{\ell}(u)$ with $\mu_{\ell}(u)=r_{\ell}(u)/f(u)$, we get 
		\begin{equation*}
		\begin{aligned}
		\left|\E[\widehat{S}_{\ell,\gamma}(x)]\right|&=\left|\frac{1}{nh} \E\left[\sum_{i=1}^{n} T_i^{\star,-\ell} (X_i-x)^\gamma K_h(X_i-x)\right] \right|\\
		&=\left|\frac{1}{h} \E\left[ T_1^{\star,-\ell} (X_1-x)^\gamma K_h(X_1-x)\right] \right|\\
		&=\left|\frac{1}{h} \int (u-x)^\gamma K_h(u-x)r_{\ell}(u)du \right|\\
		&=\left|\int (vh)^\gamma K(v)r_{\ell}(x+vh)dv \right|\\
		&=h^\gamma\left|\int v^\gamma K(v)\{r_{\ell}(x)+vh r^{\prime}_{\ell}(\xi)\}dv \right|\\
		&\leq h^\gamma r_{\ell}(x) \left|\int v^\gamma K(v)dv\right|+h^{\gamma +1}\left| \int v^{\gamma+1}  K(v) r^{\prime}_{\ell}(\xi)dv \right|.\\
		&=O\left( h^\gamma\right). \\
		\end{aligned}
		\end{equation*}
	\end{proof}
	\noindent   Now, combining  on the one hand \hyperref[lem1]{Lemma 5.1} and \hyperref[lem3]{Lemma 5.3} and on the other hand \hyperref[lem2]{Lemma 5.2} and \hyperref[lem3]{Lemma 5.3}, we conclude the proof of \hyperref[prop1]{Proposition 1}
\end{proof}
\begin{proof}[Proof of Proposition 2.]
	Let remark the decomposition for $\ell=1,2$: 
	\begin{equation*}
	\begin{aligned}
	\widetilde{\mu}_{\ell}(x)-\E[\widetilde{\mu}_{\ell}(x)]&=\widetilde{S}_{2,2}(x)\widetilde{S}_{\ell,0}(x)-\widetilde{S}_{2,1}(x)\widetilde{S}_{\ell,1}(x)-\E\left[\widetilde{S}_{2,2}(x)\widetilde{S}_{\ell,0}(x)-\widetilde{S}_{2,1}(x)\widetilde{S}_{\ell,1}(x)\right]\\
	&=\widetilde{S}_{2,2}(x)\widetilde{S}_{\ell,0}(x)-\E\left[\widetilde{S}_{2,2}(x)\widetilde{S}_{\ell,0}(x)\right] -\left\{\widetilde{S}_{2,1}(x)\widetilde{S}_{\ell,1}(x)- \E\left[\widetilde{S}_{2,1}(x)\widetilde{S}_{\ell,1}(x)\right]\right\}\\
	&=:\mathcal{E}_{\ell,1}(x)-\mathcal{E}_{\ell,2}(x).
	\end{aligned}
	\end{equation*}
	\noindent	On the one side $\{\mathcal{E}_{\ell,1},\;\text{for}\; \ell=1,2\}$, we have 
	\begin{multline}\label{decomp1}
	\mathcal{E}_{\ell,1}(x)=\left(\widetilde{S}_{2,2}(x)-\E\left[\widetilde{S}_{2,2}(x)\right]\right)\left(\widetilde{S}_{\ell,0}(x)-\E\left[\widetilde{S}_{\ell,0}(x)\right]\right)+\E\left[\widetilde{S}_{2,2}(x)\right]\left(\widetilde{S}_{\ell,0}(x)-\E\left[\widetilde{S}_{\ell,0}(x)\right]\right)\\
	+\E\left[\widetilde{S}_{\ell,0}(x)\right]\left(\widetilde{S}_{2,2}(x)-\E\left[\widetilde{S}_{2,2}(x)\right]\right)-\C\left(\widetilde{S}_{\ell,0}(x),\widetilde{S}_{2,2}(x)\right).
	\end{multline}
	\noindent On the other side $\{\mathcal{E}_{\ell,2},\;\text{for}\; \ell=1,2\}$, we have 
	\begin{multline}\label{decomp2}
	\mathcal{E}_{\ell,2}(x)=\left(\widetilde{S}_{2,1}(x)-\E\left[\widetilde{S}_{2,1}(x)\right]\right)\left(\widetilde{S}_{\ell,1}(x)-\E\left[\widetilde{S}_{\ell,1}(x)\right]\right)+\E\left[\widetilde{S}_{2,1}(x)\right]\left(\widetilde{S}_{\ell,1}(x)-\E\left[\widetilde{S}_{\ell,1}(x)\right]\right)\\
	+\E\left[\widetilde{S}_{\ell,1}(x)\right]\left(\widetilde{S}_{2,1}(x)-\E\left[\widetilde{S}_{2,1}(x)\right]\right)-\C\left(\widetilde{S}_{2,1}(x),\widetilde{S}_{\ell,1}(x)\right).
	\end{multline}
	
	\noindent It remains to study each term of the decomposition (\ref{decomp1}) and (\ref{decomp2}). We want to mention that most of the terms are studied in \hyperref[lem2]{Lemma 5.2} and \hyperref[lem3]{Lemma 5.3}. The covariance terms are studied in the two following Lemmas.
	\begin{lemma}\label{lem4}
		Under Hypotheses \hyperref[H1]{H1}, \hyperref[H2]{H2} and \hyperref[H4]{H4}, for $\ell=1,2$ and $n$ large enough, we have
		\begin{equation*}
		\C\left(\widetilde{S}_{\ell,0}(x),\widetilde{S}_{2,2}(x)\right)=\text{o}\left(\sqrt{\frac{\log n}{nh}}\right).
		\end{equation*}
	\end{lemma}
	\begin{proof}[Proof of Lemma 5.4.] By definition for $\ell=1,2$, we have
		\begin{equation*}
		\begin{aligned}
		\C\left(\widetilde{S}_{\ell,0}(x),\widetilde{S}_{2,2}(x)\right)&=\E\left[\widetilde{S}_{\ell,0}(x)\widetilde{S}_{2,2}(x)\right]-\E\left[\widetilde{S}_{\ell,0}(x)\right]\E\left[\widetilde{S}_{2,2}(x)\right].\\
		\end{aligned}
		\end{equation*}
		\noindent The proof will be made in three steps.
		\begin{itemize}
			\item [{\it Step 1.}] It is easy to see that   under \hyperref[H2]{H2} and \hyperref[H4]{H4} for $\ell=1,2$ and  using  \hyperref[lem3]{Lemma 5.3}, we get $\E\left[\widetilde{S}_{\ell,0}(x)\right]=O(1)$. Similarly, under \hyperref[H2i]{H2 i)} and \hyperref[H4]{H4} for $\ell=2$ we have \\$\E\left[\widetilde{S}_{2,2}(x)\right]=O(h^2)$.
			Now, it remains to study the quantity $\E\left[\widetilde{S}_{\ell,0}(x)\widetilde{S}_{2,2}(x)\right]$. For that, it suffices to   remark that
		\end{itemize}
		\begin{equation*}
		\begin{aligned}
		\E\left[\widetilde{S}_{\ell,0}(x)\widetilde{S}_{2,2}(x)\right]&=\frac{1}{(nh)^2}\E\left[\sum_{j=1}^{n} T_j^{\star,-\ell} K_h(X_j-x)\sum_{i=1}^{n} T_i^{\star,-2}(X_i-x)^2 K_h(X_j-x)\right]\\
		&=\frac{1}{(nh)^2}\left\{n\E\left[T_1^{\star,-\ell-2} (X_1-x)^2 K_h^2(X_1-x)\right]\right.\\
		&\left.+n(n-1)\E\left[T_1^{\star,-\ell} K_h(X_1-x)\right]\E\left[T_1^{\star,-2}(X_1-x)^2 K_h(X_1-x)\right]\right\}.
		\end{aligned}
		\end{equation*}
		\begin{itemize} 
			\item [{\it Step 2.}] 	Here after  denote by $\varrho=\ell+2,\; \text{for} \;\ell=1,\, 2$. First, we have to calculate
		\end{itemize}
		\begin{eqnarray}\label{E(T^ell)}
		\E\left[T_1^{\star,-\varrho}|X_1=u\right]&=&\E\left[\frac{\delta_1Y_1^{-\varrho}}{\overline{G}^2(Y_1)}|X_1=u\right]\nonumber\\
		&=&\E\left[\frac{T_1^{-\varrho}}{\overline{G}^2(T_1)}\E[\mathds{1}_{\{T_1 \leq R_1\}}|T_1]|X_1=u\right]\nonumber\\
		&=&\E\left[\frac{T_1^{-\varrho}}{\overline{G}(T_1)}|X_1=u\right]\nonumber\\
		& \leq& \frac{1}{\overline{G}(\tau_F)} \int t^{-\varrho}f_{T_1|X_1}(t|u)dt.
		\end{eqnarray}
		\begin{itemize} 
			\item [{\it Step 3.}] Then, using the conditional expectation property, Taylor expansion and under \hyperref[H2ii]{H2 ii)} and \hyperref[H4]{H4}, we have 
		\end{itemize}
		\begin{equation*}
		\begin{aligned}
		\E\left[T_1^{\star,-\varrho} (X_1-x)^2 K_h^2(X_1-x)\right]&=\E\left[ (X_1-x)^2 K_h^2(X_1-x)\E\left[T_1^{\star,-\varrho}|X_1\right]\right]\\
		&=\int (u-x)^2 K_h^2(u-x)\E\left[T_1^{\star,-\varrho}|X_1=u\right]f(u)du\\
		&\leq \frac{1}{\overline{G}(\tau_F)} \int (u-x)^2 K_h^2(u-x)\int t^{-\varrho}f_{T_1|X_1}(t|u)dt f(u)du\\
		&= \frac{1}{\overline{G}(\tau_F)} \int (u-x)^2 K_h^2(u-x) r_{\varrho}(u)du\\
		&= \frac{h^3}{\overline{G}(\tau_F)} \int v^2 K^2(v) r_{\varrho}(x+vh)dv\\
		&= \frac{h^3}{\overline{G}(\tau_F)} \int v^2 K^2(v) r_{\varrho}(x)dv+\frac{h^4}{\overline{G}(\tau_F)} \int v^3 K^2(v) r_{\varrho}^{\prime}(\xi)dv.\\
		\end{aligned}
		\end{equation*}		
		Finally, combining the  three steps, we get 
		\begin{equation*}
		\begin{aligned}
		\C\left(\widetilde{S}_{\ell,0}(x),\widetilde{S}_{2,2}(x)\right)
		=\text{O}\left(\frac{h}{n}\right)
		\end{aligned}
		\end{equation*}
		which is negligible with respect to $\displaystyle \sqrt{\frac{\log n}{nh}}$.
	\end{proof}
	
	\begin{lemma}\label{lem5}
		Under Hypotheses \hyperref[H1]{H1}, \hyperref[H2i]{H2 i)} and \hyperref[H4]{H4}, for $\ell=1,2$ and $n$ large enough, we have
		\begin{equation*}
		\C\left(\widetilde{S}_{\ell,1}(x),\widetilde{S}_{2,1}(x)\right)=\text{o}\left(\sqrt{\frac{\log n}{nh}}\right).
		\end{equation*}
	\end{lemma}
	
	
	\begin{proof}[Proof of Lemma 5.5.] In the same way, for $\ell=1,2$, write
		\begin{equation*}
		\begin{aligned}
		\C\left(\widetilde{S}_{\ell,1}(x),\widetilde{S}_{2,1}(x)\right)&=\E\left[\widetilde{S}_{\ell,1}(x)\widetilde{S}_{2,1}(x)\right]-\E\left[\widetilde{S}_{\ell,1}(x)\right]\E\left[\widetilde{S}_{2,1}(x)\right].\\
		\end{aligned}
		\end{equation*}We will use the following steps:
		\begin{itemize}
			\item [Step 4.]  It is easy to see that from \hyperref[lem3]{Lemma 5.3} under \hyperref[H2]{H2} and \hyperref[H4]{H4} for $\ell=1,2$ we have $\E\left[\widetilde{S}_{\ell,1}(x)\right]=O(h)$. Similarly, under \hyperref[H2i]{H2 i)} and \hyperref[H4]{H4} for $\ell=2$ we have $\E\left[\widetilde{S}_{2,1}(x)\right]=O(h)$.	
			Now, it remains to study the quantity $\E\left[\widetilde{S}_{\ell,1}(x)\widetilde{S}_{2,1}(x)\right]$. To do that, Let us remark that
			\begin{equation*}
			\begin{aligned}
			\E\left[\widetilde{S}_{\ell,1}(x)\widetilde{S}_{2,1}(x)\right]&=\frac{1}{(nh)^2}\E\left[\sum_{j=1}^{n} T_j^{\star,-\ell} (X_j-x) K_h(X_j-x)\sum_{i=1}^{n} T_i^{\star,-2}(X_i-x) K_h(X_i-x)\right]\\
			&=\frac{1}{(nh)^2} \left\{n\E\left[T_1^{\star,-\ell-2} (X_1-x)^2 K_h^2(X_1-x)\right]\right.\\
			&\left.+n(n-1)\E\left[T_1^{\star,-\ell} (X_1-x) K_h(X_1-x)\right]\times \E\left[T_1^{\star,-2}(X_1-x) K_h(X_1-x)\right]\right\}.\\
			\end{aligned}
			\end{equation*}
			\item [Step 5.]   We use the same notation $\varrho=\ell+2$ to avoid any confusion and by  (\ref{E(T^ell)}) we have 
			\begin{equation*}
			\begin{aligned}
			\E\left[T_1^{\star,-\varrho} (X_1-x)^2 K_h^2(X_1-x)\right]&=\E\left[ (X_1-x)^2 K_h^2(X_1-x)\E\left[T_1^{\star,-\varrho}|X_1\right]\right]\\
			&=\int (u-x)^2 K_h^2(u-x)\E\left[T_1^{\star,-\varrho}|X_1=u\right]f(u)du\\
			&= \frac{1}{\overline{G}(\tau_F)} \int (u-x)^2 K_h^2(u-x)\int t^{-\varrho}f_{T_1,X_1}(t,u)dt du\\
			&= \frac{1}{\overline{G}(\tau_F)} \int (u-x)^2 K_h^2(u-x) r_{\varrho}(u)du\\
			&= \frac{h}{\overline{G}(\tau_F)} \int (vh)^2 K^2(v) r_{\varrho}(x+vh)dv\\
			&= \frac{h^3}{\overline{G}(\tau_F)} \int v^2 K^2(v) r_{\varrho}(x+vh)dv,\\
			\end{aligned}
			\end{equation*}
		\end{itemize} 
		\noindent 	under \hyperref[H2ii]{H2 ii)} and \hyperref[H4]{H4}, we get O$(h^3)$.
		
		\noindent 	    Combining steps 4 and 5 we have 
		\begin{equation*}
		\begin{aligned}
		\C\left(\widetilde{S}_{\ell,1}(x),\widetilde{S}_{2,1}(x)\right)
		=\text{O}\left(\frac{h}{n}\right)
		\end{aligned}
		\end{equation*}
		which is negligible with respect to $\displaystyle \sqrt{\frac{\log n}{nh}}$.
	\end{proof}
	\noindent	Finally, combining \hyperref[lem2]{Lemma 5.2} and \hyperref[lem3]{Lemma 5.3} in the proof of \hyperref[prop1]{Proposition 1} with  \hyperref[lem4]{Lemma 5.4} and \hyperref[lem5]{Lemma 5.5}, we get the result of the \hyperref[prop2]{Proposition 2}.
\end{proof}
\begin{proof}[Proof of Proposition 3.] Using the conditional expectation property for  $\ell=1,\, 2$ \\   we get  $\E[T^{\star,-2}_1T^{\star,-\ell}_2|X_1,X_2]=\mu_{2}(X_1) \mu_{\ell}(X_2)$. Then 
	we have
	\begin{equation*}
	\begin{aligned}
	\E[\widetilde{\mu}_\ell(x)]-r_\ell(x)r_2(x)&=\E\left[\widetilde{S}_{2,2}(x)\widetilde{S}_{\ell,0}(x)-\widetilde{S}_{2,1}(x)\widetilde{S}_{\ell,1}(x)\right]-r_\ell(x)r_2(x)\\
	&\hspace*{-.7in} =\frac{1}{(nh)^2}\E\left[ \left( \sum_{i=1}^n (X_i-x)^2 T^{\star,-2}_iK_h(X_i-x)\right)\left( \sum_{j=1}^n T^{\star,-\ell}_jK_h(X_j-x)\right)\right.\\
	&\hspace*{-.7in} \left.-\left( \sum_{i=1}^n (X_i-x) T^{\star,-2}_iK_h(X_i-x)\right)\left( \sum_{j=1}^n T^{\star,-\ell}_j (X_j-x)K_h(X_j-x)\right)\right]-r_\ell(x)r_2(x)\\
	&\hspace*{-.7in}=h^{-2}\E\left[ \left( (X_1-x)^2 T^{\star,-2}_1K_h(X_1-x)\right)\left( T^{\star,-\ell}_2K_h(X_2-x)\right)\right.\\
	&\hspace*{-.7in}-\left. \left( (X_1-x) T^{\star,-2}_1K_h(X_1-x)\right)\left( T^{\star,-\ell}_2 (X_2-x)K_h(X_2-x)\right)\right]-r_\ell(x)r_2(x) \\
	&\hspace*{-.7in}=h^{-2} \int \int K(u-x)K(v-x) \left( (u-x)^2-(u-x)(v-x)\right)\\
	&\hspace*{-.7in}\times \E[T^{\star,-2}_1T^{\star,-\ell}_2|X_1=u,X_2=v]f(u)f(v)dudv-r_\ell(x)r_2(x)\\
	&\hspace*{-.7in}=h^{-2} \int \int K(u-x)K(v-x) \left( (u-x)^2-(u-x)(v-x)\right)r_2(u)r_\ell(v)dudv-r_\ell(x)r_2(x).\\
	\end{aligned}
	\end{equation*}
	By a change of variable, we get 
	\begin{equation*}
	\begin{aligned}
	&\hspace*{.5in}=\int \int K(t)K(s) \left( (th)^2-(th)(sh)\right)r_2(x+th)r_\ell(x+sh)dtds-r_\ell(x)r_2(x)\\
	&\hspace*{.5in}=h^2 \int \int (t^2-ts) K(t)K(s)r_2(x+th)r_\ell(x+sh)dtds-r_\ell(x)r_2(x)\\
	&\hspace*{.5in}=h^2 \int \int (t^2-ts) K(t)K(s)\{r_2(x+th)r_\ell(x+sh)-r_\ell(x)r_2(x)\}dtds\\
	\end{aligned}
	\end{equation*}
	Using Taylor expansion, it is easy to see for $\ell=1,2$	
	\begin{equation*}
	\begin{aligned}
	r_2(x+th)r_\ell(x+sh)-r_\ell(x)r_2(x)&=(r_2(x+th)-r_2(x))(r_\ell(x+sh)-r_\ell(x))\\
	&+r_\ell(x)(r_2(x+th)-r_2(x))+r_2(x)(r_\ell(x+sh)-r_\ell(x))\\
	&=h^2 t s r^{\prime}_2(\xi_1)r^{\prime}_\ell(\xi_2)+thr^{\prime}_2(\xi_1)r_\ell(x)+shr^{\prime}_\ell(\xi_2)r_2(x)
	\end{aligned}
	\end{equation*}
	\noindent where $\xi^\prime_1 \in ]x,x+th[$ and  $\xi^\prime_2 \in ]x,x+sh[$. Then, we have 
	\begin{equation*}
	\begin{aligned}
	\E[\widetilde{\mu}_\ell(x)]-r_\ell(x)r_2(x)&=h^2 \int \int (t^2-ts) K(t)K(s)\{h^2 t s r^{\prime}_2(\xi_1)r^{\prime}_\ell(\xi_2)+thr^{\prime}_2(\xi_1)r_\ell(x)+shr^{\prime}_\ell(\xi_2)r_2(x)\}dt ds\\
	&\hspace*{-.7in}=h^4 \int\int ts(t^2-ts) K(t)K(s)r^{\prime}_2(\xi_1)r^{\prime}_\ell(\xi_2)dt ds+h^3\int\int t(t^2-ts) K(t)K(s)r^{\prime}_2(\xi_1)r_\ell(x)dt ds\\
	&\hspace*{-.7in}+h^3\int\int s(t^2-ts) K(t)K(s)r^{\prime}_\ell(\xi_2)r_2(x)dt ds\\
	&\hspace*{-.7in}=h^4 \left(\left(\int t^3K(t)r^{\prime}_2(\xi_1)dt\right)\left(\int sK(s)r^{\prime}_\ell(\xi_2)ds\right)-\left(\int t^2 K(t)r^{\prime}_2(\xi_1)dt\right)\left(\int s^2 K(s) r^{\prime}_\ell(\xi_2)ds\right)\right)\\
	&\hspace*{-.7in}+h^3 \left(\left(r_\ell(x)\int t^3K(t)r^{\prime}_2(\xi_1)dt\right)\left(\int K(s)ds\right)-\left(r_\ell(x)\int s K(s)ds\right)\left(\int t^2 K(t) r^{\prime}_2(\xi_1)dt\right)\right)\\
	&\hspace*{-.7in}+h^3 \left(\left(r_2(x)\int t^2K(t)dt\right)\left(\int sK(s)r^{\prime}_\ell(\xi_2)ds\right)-\left(r_2(x)\int t K(t)dt\right)\left(\int s^2 K(s) r^{\prime}_\ell(\xi_2)ds\right)\right).\\
	\end{aligned}
	\end{equation*}
	\noindent Under the hypotheses \hyperref[H2i]{H2 i)} and \hyperref[H4]{H4} for $\ell=1,2$, the result can be deduced directly from the last equality $\text{O}(h^3)$.
\end{proof}
\section{Appendix}
\noindent {\bf Corollary A.8.} in \cite{Ferraty2006} p. 234 \label{coro}. \\ Let $U_i$ be a sequence of independent r.v. with zero mean. If $\forall \; m \geq 2$, $\exists \; C_m>0$, $\E[|U_1^m|] \leq C_m a^{2(m-1)}$, we have 
	\begin{equation*}
	\forall \varepsilon>0, \;\;\; \mathbb{P}\left( \left|\sum_{i=1}^{n} U_i\right|> n\varepsilon\right) \leq 2 \exp \left\{-\frac{\varepsilon^2 n}{2a^2(1+\varepsilon)}\right\}.
	\end{equation*}
\section{Conclusion}
In this paper we establish the uniform strong consistency with rate  for  the local linear relative error regression estimator over a compact set, when the variable of interest is subject to random right censoring. A large simulation study was conducted through which our estimator performance was highlighted in spite of well known boundary effects of kernel estimation. On the one hand, for a practical point of view the results indicate the lack of flexibility in estimating a function using traditional approaches. On the other hand, the proposed estimates are closest to the true curve. In conclusion, the LLRER method has more advantage than the CR and LLCR such as the efficiency in presence of outliers and censorship compared to the two other methods. Finally, we point out that the bias term  appears to inhabit, however the combination of the two methods LL and RER has revealed several terms which do not allow to obtain a standard result of order one or two. Conversely, we can say that the reduction of the  bias is highlighted.


\begin{thebibliography}{99}
	
	\bibitem[\protect\citeauthoryear{Attouch {\it et al.}}{2017}]{Attouch2017} Attouch ,M., Laksaci A. and Messabihi, N. (2017). Nonparametric relative error regression for spatial random variables. {\it Statist. Papers} 58, 987--1008.
	
	
	\bibitem[\protect\citeauthoryear{Beran}{1981}]{Beran1981} Beran, M. (1981). Nonparametric regression with randomly censored survival data. { Technical report, Dept. Statist., Univ. of California, Berkeley.}
	
	
	
	\bibitem[\protect\citeauthoryear{Carbonez {\it et al.}}{1995}]{Carbonez1995} Carbonez, A., Gyorfi, L. and Van Der Meulen, E.C. (1995).  Partitioning estimates of a regression function under random censoring. {\it Statist. and Decisions.} 76, 1335--1344.
	
	\bibitem[\protect\citeauthoryear{Chahad {\it et al.}}{2017}]{Chahad2017} Chahad, A.,  Ait-Hennani, L. and Laksaci, A. (2017). Functional local linear estimate for functional relative error regression. {\it J. of Statist. Theory and Practice} 76, 1335--1344. 
	
	\bibitem[\protect\citeauthoryear{Chen {\it et al.}}{2010}]{Chen2010}  Chen, K., Guo, S., Lin, Y. and Ying, Z. (2010). Least absolute relative error estimation.
	{\it J. Amer. Statist. Assoc.} 105, 1104--1112.
	
	\bibitem[\protect\citeauthoryear{Dabrowska}{1987}]{Dabrowska1987} Dabrowska, D. (1987). Nonparametric regression with censored survival data. {\it Scand. J. Statist.} { 14,} 181--197.
	
	\bibitem[\protect\citeauthoryear{Dabrowska}{1989}]{Dabrowska1989} Dabrowska, D. (1989). Uniform consistency of the kernel conditional Kaplan-Meier estimate. {\it Ann. of Statist.} { 17,} 1157--1167.
	
	\bibitem[\protect\citeauthoryear{Deuheuvels and Einmahl}{2000}]{Deuheuvels2000} Deuheuvels, P. and Einmahl, J. H. (2000). Functional limit laws for the increments of Kaplan-Meier product limit processes and applications. {\it  Ann Probab.} 28, 1301--1335.
	
	\bibitem[\protect\citeauthoryear{El Ghouch and Van Keilegom}{2008}]{ElGhouch2008} El Ghouch, A. and Van Keilegom, I. (2008). Nonparametric regression with rependent rensored data. {\it Scandinavian J. of Statist.} 35(2), 228--247.
	
	\bibitem[\protect\citeauthoryear{El Ghouch and Van Keilegom}{2009}]{ElGhouch2009} El Ghouch, A. and Van Keilegom, I. (2009). Local linear quantile regression with dependent censored data. {\it Statist. Sinica}. 19, 1621--1640.  
	
	\bibitem[\protect\citeauthoryear{Fan}{1992}]{Fan1992} Fan, J. (1992). Design adaptative nonparametric regression. {\it  Ann Probab.} 87, 998--1004.
	
	\bibitem[\protect\citeauthoryear{Fan and Gijbels}{1996}]{Fan1996} Fan, J. and Gijbels, I. (1996). Local polynomial modeling and its applications. Chapman \& Hall/CRC, New York.
	
	\bibitem[\protect\citeauthoryear{Fan and Yao}{2003}]{Fan2003} Fan, J. and Yao, Q. (2003). Nonlinear time series: nonparametric and parametric methods. Springer, New York.
	
	\bibitem[\protect\citeauthoryear{Ferraty and Vieu}{2006}]{Ferraty2006} Ferraty, F. and Vieu, P. (2006). Nonparametric functional data analysis: theory and practice. Springer, New York.
	
	\bibitem[\protect\citeauthoryear{Guessoum and Ould Sa\"id}{2008}]{Guessoum2008} Guessoum, Z. and Ould Sa\"id, E. (2008). On nonparametric estimation of the regression function under random censorship model. {\it  Statist. and Decisions} 26, 1001--1020.
	
	\bibitem[\protect\citeauthoryear{Hirose and Masuda}{2018}]{Hirose2018}
	Hirose, K, Masuda, H. (2018). Robust relative error estimation. {\it Entropy}. 20. Paper No. 632, 24 pp.
	
	\bibitem[\protect\citeauthoryear{Hu}{2019}]{Hu2019} Hu, D.H. (2019).
	Local least product relative error estimation for varying coefficient multiplicative regression model. {\it Acta. Math. Appl. Sinica.} 35, 274--286.
	
	\bibitem[\protect\citeauthoryear{Jones \it{et al.}}{2008}]{Jones2008} Jones, M. C., Park, H., Shin, K. I., Vines, S. K. and Jeong, S. O. (2008). Relative error prediction via kernel regression smoothers. {\it  Journal of Statist. Plann. and Infer.} 138, 2887--2898.
	
	\bibitem[\protect\citeauthoryear{Kaplan and Meier}{1958}]{Kaplan1958} Kaplan, E. L. and Meier, P. (1958). Nonparametric estimation from incomplete observations. {\it J. Amer. Stat. Assoc.} 53, 458--481.
	
	\bibitem[\protect\citeauthoryear{Klein and Moeschberger}{2006}]{Klein2006} Klein, J. P. and Moeschberger, M. L. (2006). Survival analysis: techniques for censored
	and truncated data. Springer Science Business Media.
	
	\bibitem[\protect\citeauthoryear{Kohler {\it et al.}}{2002}]{Kohler2002} Kohler, M., M\'ath\'e, K. and Pint\'er, M. (2002). Prediction from randomly right censored data. {\it J. Multivar. Anal.} 80, 73--100.
	
	\bibitem[\protect\citeauthoryear{Nadaraya}{1964}]{Nadaraya1964} Nadaraya, E. A. (1964). On estimating regression. {\it Theor. Probab. Appl.} 9, 141--142.
	
	
	\bibitem[\protect\citeauthoryear{Park and Stefanski}{1998}]{Park1998} Park, H. and Stefanski, L. A. (1998). Relative error prediction. {\it Statist. \& Probab. Lett.} 40, 227--236.
	
	\bibitem[\protect\citeauthoryear{Port}{1994}]{Port1994} Port, S.C. (1994). Theoretical Probability for Applications. John Wiley \& Sons Inc.
	
	
	
	\bibitem[\protect\citeauthoryear{Thiam}{2018}]{Thiam2018} Thiam, B. (2018). Relative error prediction in nonparametric deconvolution regression model. {\it Statist. Neerlandica}. 73, 63--77.
	
	\bibitem[\protect\citeauthoryear{Watson}{1964}]{Watson1964} Watson, G.S. (1964). Smooth regression analysis. {\it Sankhy\`a.} 26, 359--372.
\end{thebibliography}
\end{document}